\documentclass[a4paper, reqno, 11pt]{amsart}

\usepackage[usenames,dvipsnames]{color}
\usepackage{amsthm,amsfonts,amssymb,amsmath,amsxtra}
\usepackage[all]{xy}
\SelectTips{cm}{}
\usepackage{xr-hyper}
\usepackage[colorlinks=
   citecolor=Black,
   linkcolor=Red,
   urlcolor=Blue]{hyperref}
\usepackage{verbatim}
\usepackage{dynkin-diagrams}  
\usepackage{tikz}
\usepackage{array}
\usepackage[capitalise]{cleveref}
\newcommand{\vertex}{\node[vertex]}
\tikzstyle{vertex}=[circle,fill=black!25,minimum size=20pt,inner sep=0pt]

\usepackage{longtable}

\usepackage{enumerate}
\renewcommand{\theenumi}{\roman{enumi}}
\usepackage{adjustbox}

\usepackage[margin=1.25in]{geometry}
\usepackage{mathrsfs}

\RequirePackage{xspace}
\RequirePackage{etoolbox}
\RequirePackage{varwidth}
\RequirePackage{enumitem}
\RequirePackage{tensor}
\RequirePackage{mathtools}
\RequirePackage{longtable}
\RequirePackage{multirow}

\def\<{\langle}
\def\>{\rangle}

\newcommand{{\BG}}{\ensuremath{\mathbb {G}}\xspace}

\newcommand{{\BK}}{\ensuremath{\mathbb {K}}\xspace}

\let\Im\relax
\DeclareMathOperator{\Im}{Im}

\DeclareMathOperator{\Span}{span}
\newtheorem*{theorem*}{Theorem}
\newtheorem{theorem}{Theorem}
\newtheorem{proposition}[theorem]{Proposition}
\newtheorem{lemma}[theorem]{Lemma}

\newtheorem{corollary}[theorem]{Corollary}

\theoremstyle{definition}
\newtheorem{definition}[theorem]{Definition}
\newtheorem{example}[theorem]{Example}

\newtheorem{remark}[theorem]{Remark}

\newtheorem{notation}[theorem]{Notation}

\newtheorem{algorithm}[theorem]{Algorithm}
\newtheorem{assumption}[theorem]{Assumption}

\numberwithin{equation}{section}
\numberwithin{theorem}{section}

\usepackage{url}
\usepackage{hyperref}

\usepackage{array}
\newcolumntype{P}[1]{>{\centering\arraybackslash}p{#1}}
\newcolumntype{M}[1]{>{\centering\arraybackslash}m{#1}}

\setitemize[0]{leftmargin=*,itemsep=\the\smallskipamount}
\setenumerate[0]{leftmargin=*,itemsep=\the\smallskipamount}

\renewcommand{\to}{%
   \ifbool{@display}{\longrightarrow}{\rightarrow}%
   }
\let\shortmapsto\mapsto
\renewcommand{\mapsto}{%
   \ifbool{@display}{\longmapsto}{\shortmapsto}%
   }
\newlength{\olen}
\newlength{\ulen}
\newlength{\xlen}
\newcommand{\xra}[2][]{%
   \ifbool{@display}%
      {\settowidth{\olen}{$\overset{#2}{\longrightarrow}$}%
       \settowidth{\ulen}{$\underset{#1}{\longrightarrow}$}%
       \settowidth{\xlen}{$\xrightarrow[#1]{#2}$}%
       \ifdimgreater{\olen}{\xlen}%
          {\underset{#1}{\overset{#2}{\longrightarrow}}}%
          {\ifdimgreater{\ulen}{\xlen}%
             {\underset{#1}{\overset{#2}{\longrightarrow}}}
             {\xrightarrow[#1]{#2}}}}%
      {\xrightarrow[#1]{#2}}
   }
\makeatother
\newcommand{\xyra}[2][]{%
   \settowidth{\xlen}{$\xrightarrow[#1]{#2}$}%
   \ifbool{@display}%
      {\settowidth{\olen}{$\overset{#2}{\longrightarrow}$}%
       \settowidth{\ulen}{$\underset{#1}{\longrightarrow}$}%
       \ifdimgreater{\olen}{\xlen}%
          {\mathrel{\xymatrix@M=.12ex@C=3.2ex{\ar[r]^-{#2}_-{#1} &}}}%
          {\ifdimgreater{\ulen}{\xlen}%
             {\mathrel{\xymatrix@M=.12ex@C=3.2ex{\ar[r]^-{#2}_-{#1} &}}}
             {\mathrel{\xymatrix@M=.12ex@C=\the\xlen{\ar[r]^-{#2}_-{#1} &}}}}}%
      {\mathrel{\xymatrix@M=.12ex@C=\the\xlen{\ar[r]^-{#2}_-{#1} &}}}%
   }
\makeatletter
\newcommand{\xla}[2][]{%
   \ifbool{@display}%
      {\settowidth{\olen}{$\overset{#2}{\longleftarrow}$}%
       \settowidth{\ulen}{$\underset{#1}{\longleftarrow}$}%
       \settowidth{\xlen}{$\xleftarrow[#1]{#2}$}%
       \ifdimgreater{\olen}{\xlen}%
          {\underset{#1}{\overset{#2}{\longleftarrow}}}%
          {\ifdimgreater{\ulen}{\xlen}%
             {\underset{#1}{\overset{#2}{\longleftarrow}}}
             {\xleftarrow[#1]{#2}}}}%
      {\xleftarrow[#1]{#2}}
   }
\newcommand{\isoarrow}{%
   \ifbool{@display}{\overset{\sim}{\longrightarrow}}{\xrightarrow\sim}%
   }
\makeatletter
\newcommand*{\defeq}{\mathrel{\rlap{%
                     \raisebox{0.3ex}{$\m@th\cdot$}}%
                     \raisebox{-0.3ex}{$\m@th\cdot$}}%
                     =}
\makeatother

\begin{document}

\date{\today}

\title[]{Regularity of Unipotent Elements in Total Positivity}
    \author{Haiyu Chen}
    \address[H. Chen]{The Institute of Mathematical Sciences and Department of Mathematics, The Chinese University of Hong Kong, Shatin, N.T., Hong Kong.}
    \email{hychen@math.cuhk.edu.hk}

    \author{Kaitao Xie}
        \address[K. Xie]{Department of Mathematics, The University of Hong Kong, Pokfulam, Hong Kong.}
        \email{kaitaoxie@connect.hku.hk}
\thanks{}


\date{\today}

\begin{abstract}
Let $G$ be a connected reductive group split over $\mathbb R$. We show that every unipotent element in the totally nonnegative monoid of $G$ is regular in some Levi subgroups, confirming a conjecture of Lusztig. 
\end{abstract}

\maketitle

\section{Introduction}
Let $G$ be a connected reductive group over $\mathbb C$, split over $\mathbb R$. Let $G_{\geq 0}$ be the totally nonnegative submonoid attached to a pinning of $G$ defined by Lusztig in \cite{L94}. In this paper, we are interested in the interaction between $G_{\geq 0}$ and the unipotent conjugacy classes in $G$. Here, we review some previous work on the related topic.

\begin{itemize}
    \item  The set of totally nonnegative unipotent elements decomposes into cells indexed by a pair of Weyl group elements $(w_1,w_2)$ with certain condition \cite[Theorem 6.6]{L94}, which we refer to as totally nonnegative unipotent cells in the sequel.
    \item He and Lusztig showed that, for each totally nonnegative unipotent cell, its intersection with the regular unipotent class is either empty or the whole cell \cite[Lemma 1.1]{HeLus22}. The cells that consist of regular unipotent elements are determined. 
\end{itemize}

The main result of this paper is the following theorem.

\begin{theorem*}
The totally nonnegative unipotent cell indexed by $(w_1,w_2)$ lies in a single conjugacy class which only depends on the supports of $w_1$ and $w_2$. This conjugacy class consists of regular unipotent elements of a family of conjugated Levi subgroups of $G$. Moreover, \cref{algorithm: weight_alg} provides a way to determine these Levi subgroups.
\end{theorem*}

In particular, every totally nonnegative unipotent element is a regular element in some Levi subgroups. This confirms a conjecture in \cite[\S9.2]{L19}. As a consequence of this result, the main conjecture in \cite[\S9.2]{L19} and the conjecture \cite[Conjecture 3.1 (1)]{HeLus22} are equivalent. 

To prove this result, we first show that each totally nonnegative unipotent cell lies in a single conjugacy class $\mathcal C_{(J_1,J_2)}$ which only depends on the supports $J_1=\text{supp}(w_1)$ and $J_2=\text{supp}(w_2)$, rather than $w_1$ and $w_2$. Then, we pick up an element $x\in \mathcal C_{(J_1,J_2)}$ and conjugate it to another element $x_J$ that is regular in a standard Levi subgroup $L_J$. Here, $J$ is a subset of simple roots. This is done in \cref{section:AD} for types $A$ and $D$, and \cref{section:E} for $E$. For non-simply-laced types, we obtain the result by folding method (\cref{section:folding}). 

Let $I$ be the set of simple roots that is compatible with the chosen pinning. Our result allows one to define a map 
$$\mathcal F: \{(J_1,J_2)\mid J_1,J_2\subseteq I,J_1\cap J_2=\emptyset\}\to \{J\mid  J\subseteq I\}/\sim$$
where $\sim$ is an equivalent relation on the set $\{J|J\subseteq I\}$ defined by $J\sim J^\prime$ if there is $w\in W$ such that $J = w \cdot J^\prime$. Here, the conjugacy class $\mathcal C_{(J_1,J_2)}$ contains the regular unipotent elements in the Levi subgroup $L_J$, where $J$ is a representative in $\mathcal F(J_1,J_2)$. We provide an algorithm to calculate the image of $(J_1,J_2)$ under $\mathcal F$ (\cref{section:weight}). We use a computer program to verify that this algorithm indeed gives $\mathcal F(J_1,J_2)$ in the exceptional types.

\subsection{Acknowledgement}

We truly appreciate Xuhua He for his lecture on total positivity where we were introduced to this wonderful area and his valuable suggestions. We would like to express our gratitude to Syx Pek and Shiwei Isaac Koh for many helpful early discussions. We would like to thank George Lusztig and Felix Schremmer for helpful comments and suggestions.

\section{Preliminaries} 

We begin by introducing the basic setup and results of total positivity. We then give a brief introduction to the classification of unipotent conjugacy classes due to Bala and Carter, followed by a review of some basic facts in Lie theory that will be used frequently in the sequel.

\subsection{Total Positivity} 

We review some basic results from \cite{L94}. Let $G$ be a connected reductive group over $\mathbb C$, split over $\mathbb R$, with the set of unipotent elements denoted by $\mathcal U$. Fix a pinning $\underline{\mathbf P} = (T,B^+,B^-,u_{\alpha_i},u_{-\alpha_i};\alpha_i\in I)$ and let $G_{\geq 0}$ be the totally nonnegative submonoid of $G$ associated to $\underline{\mathbf P}$. Here, $I$ is the set of simple roots. Denote $W = N_G(T)/T$ the Weyl group.

\begin{definition}
For each $w\in W$, choose a sequence $(\alpha_{i_1},\alpha_{i_2},...,\alpha_{i_n})$ of simple roots such that $s_{\alpha_{i_1}}s_{\alpha_{i_2}}...s_{\alpha_{i_n}}$ is a reduced expression of $w$. Define the \emph{cells}
$$U^+(w)=\{u_{\alpha_{i_1}}(a_1)u_{\alpha_{i_2}}(a_2)...u_{\alpha_{i_n}}(a_n)\mid  (a_1,a_2,...,a_n)\in\mathbb R^n_{>0}\} ,$$
$$U^-(w)=\{u_{-\alpha_{i_1}}(a_1)u_{-\alpha_{i_2}}(a_2)...u_{-\alpha_{i_n}}(a_n)\mid (a_1,a_2,...,a_n)\in\mathbb R^n_{>0}\}.$$
\end{definition}

It is shown that $U^+(w)$ and $U^-(w)$ are independent of the choice of reduced expressions of $w$ in \cite[Proposition 2.7]{L94}.

For each $w\in W$, denote the \emph{support} of $w$ to be 
$$\text{supp}(w)=\{\alpha\in I\mid  s_{\alpha}\text{ occurs in some reduced expressions of }w\}.$$
By Matsumoto's theorem, if $s_\alpha\in w$ occurs in some reduced expression, it occurs in any reduced expressions.

The following result provides a parametrization of unipotent elements in $G_{\geq 0}$.

\begin{proposition}\cite[Theorem 6.6]{L94}  \label{prop:cell}
The set $G_{\geq 0}\cap\mathcal U$ has a decomposition into cells: $$G_{\geq 0}\cap\mathcal U=
\underset{(w_1,w_2)\in Z}{\bigsqcup} U^+(w_1)U^-(w_2)$$
where the index set is 
$$Z = \{(w_1,w_2)\in W\times W\mid \text{supp}(w_1)\cap\text{supp}(w_2)=\emptyset\}.$$ 
\end{proposition}

\subsection{Classification of unipotent conjugacy classes}  
We review the classification of unipotent classes following the textbook of Carter \cite{Ca}. 

A closed connected subgroup $P$ of $G$ that contains a Borel subgroup of $G$ is called a \emph{parabolic subgroup}. For such a subgroup, there is a Levi decomposition $P = L_P\cdot U_P$, where $L_P$ is a reductive group, called a \emph{Levi subgroup}, and $U_P$ is the unipotent radical of $P$. Fix a maximal torus $T$ and Borel subgroup $B^+$ of $G$ containing $T$, one chooses a set of simple roots, denoted by $I$. The parabolic subgroups that contain $B^+$ are in correspondence with subsets of $I$. Let $P_J^+$ be the parabolic subgroup of $G$ associated with the subset $J$ of $I$, and $P^-_J$ be its opposite. We call $L_J:=P^+_J\cap P^-_J$ the \emph{standard} Levi subgroup of $G$ associated with $J$.

For a unipotent conjugacy class $\mathcal C$ of $G$, there is a Levi subgroup $L$ of $G$ such that it is minimal among the Levi subgroups that intersect $\mathcal C$. For any $\mathcal C$, the minimal Levi subgroup is well-defined up to conjugacy. We say that $L$ is a \emph{minimal} Levi subgroup of $\mathcal C$. We say that $P = L_P\cdot U_P$ is a \emph{distinguished} parabolic subgroup of $G$ if $\dim L_P = \dim U_P/(U_P,U_P)$, where $(U_P,U_P)$ is the derived subgroup of $U_P$.

The below theorem, due to Bala and Carter, gives a classification of unipotent conjugacy classes.

\begin{theorem}
There is a natural bijection 

\begin{tabular}{c c p{2cm}}
$\left\{\text{unipotent conjugacy classes of $G$}\right\}$ & $\longleftrightarrow$ & $\left\{\text{$G$-conjugacy classes of pairs $(L,P_L)$}\right\}$
\end{tabular}\\
where $L$ is a Levi subgroup of $G$ and $P_L$ is a distinguished parabolic subgroup of the derived subgroup $(L,L)$ of $L$.
\end{theorem}

Under this bijection, the pair $(L,P_L)$ corresponding to a unipotent conjugacy class $\mathcal C$ is called the Bala-Carter data of $\mathcal C$, and $L$ is a minimal Levi subgroup of $\mathcal C$.

In this paper, we are mainly interested in the unipotent conjugacy classes $\mathcal C$ where there is a Levi subgroup $L$ of $G$ such that every element in $\mathcal C\cap L$ is regular in $L$. Here we recall some facts about regular unipotent elements.

\begin{itemize}
    \item The set of regular unipotent elements in $G$ forms a single conjugacy class. Its Bala-Carter data is $(G,B)$, where $B$ is a Borel subgroup. As a consequence, for every Levi subgroup $L$ of $G$, there is a unique unipotent conjugacy class $\mathcal C$ of $G$ such that every element in $\mathcal C\cap L$ is regular in $L$. The Bala-Carter data of $\mathcal C$ is $(L,B_L)$ where $B_L$ is a Borel subgroup of the derived subgroup of $L$.
    \item A unipotent element $u\in G$ is regular if and only if there is a unique Borel subgroup $B$ of $G$ such that $u\in B$.
    \item A unipotent element $u\in G$ is regular if and only if it is conjugate to an element $\prod_{\alpha\in\Phi^+}u_\alpha(a_\alpha)$, where $a_\alpha\neq 0$ whenever $\alpha$ is a simple root.
\end{itemize}

Compared to Bala-Carter theory, the conjecture we are about to prove states that the totally nonnegative monoid only intersects with ``nice'' unipotent classes, namely, the bijection above restricts to 

\begin{tabular}{l l l}
  \bigg\{ \parbox{5cm}{\textit{unipotent conjugacy classes that intersect } $G_{\geq 0}$}  \bigg\}
  &
  $\longleftrightarrow$
  &
  $\left\{\textit{$G$-conjugacy classes of Levi subgroups} \right\}.$\\
\end{tabular}\\

\subsection{Basics on Lie algebras and adjoint action} A general reference for this section is \cite{Kn}.
Let $\mathfrak g$ be the Lie algebra of $G$. Denote by $\mathfrak{g}^{\text{reg}}$ the set of all regular elements in $\mathfrak g$, and similar notation for $G^{\text{reg}}$. Let $U^+$ be the unipotent radical of $B^+$, with Lie algebra $\mathfrak u^+$; similar notations for $U^-$ and $\mathfrak u^-$. For a standard Levi subgroup $L_J$, denote its Lie algebra by $\mathfrak{l}_J$.

The exponential map $\text{Exp}:\mathfrak g\to G$ gives rise to a bijection between unipotent conjugacy classes in $G$ and the nilpotent orbits in $\mathfrak g$. Moreover, it is $G$-equivariant and sends $\mathfrak l_J^{\text{reg}}\cap u^+$ to $L_J^{\text{reg}}\cap U^+$ and $\mathfrak l_J^{\text{reg}}\cap\mathfrak u^-$ to $L_J^{\text{reg}}\cap U^-$. Thus, the nilpotent orbit that contains regular nilpotent elements in $\mathfrak l_J$ corresponds to the unipotent conjugacy class that contains regular nilpotent elements in $L_J$ under this bijection.

Finally, we review the adjoint action of some particular elements on the root vectors. Let $\Phi$ be the root system of $G$ with the set of positive roots denoted by $\Phi^+$. We fix a root group realization $\{u_\alpha:\mathbf G_a\to G\mid \alpha\in\Phi\}$ that is compatible with the chosen pinning, and a set of \emph{root vectors} $\{X_\alpha\mid \alpha\in\Phi\}$ in $\mathfrak g$. Here, we review some well-known facts about the adjoint action.

\begin{itemize} 
    \item Let $w\in W$. Denote $\dot w\in G$ to be one of its representatives. Let $X_\alpha$ be a root vector with respect to root $\alpha\in\Phi$. Then $$\text{Ad}(\dot w)X_\alpha = c_{\dot{w},\alpha}X_{w\cdot \alpha}$$ for some nonzero constant $c_{\dot{w},\alpha}$.
    
    \item Let $\alpha,\beta\in\Phi$ with $\alpha\neq-\beta$. There exist constants $c_{\alpha,\beta,i}$ such that $$\text{Ad}(u_\alpha(a))X_\beta = X_\beta+\sum_{\beta+i\alpha\in\Phi, i>0}c_{\alpha,\beta,i}a^iX_{\beta+i\alpha}.$$

    Moreover, $c_{\alpha,\beta,i}\neq 0$ whenever $\beta+i\alpha\in\Phi$.
\end{itemize}

\section{Main result}
In this section, we present the main result of this paper, which determines the unipotent conjugacy class of a unipotent element in $\mathcal U\cap G_{\geq 0}$ according to the cell decomposition in \cref{prop:cell}. Also, we prove some lemmas to reduce the problem in consideration.

\subsection{The fusion map} 
Let $I$ be the set of all simple roots. Let $U^+(w_1)U^-(w_2)$ be a cell in \cref{prop:cell}.

\begin{lemma} \label{lem:uniqueness}
 Let  $J_1,J_2$ be subsets of $I$ such that $J_1\cap J_2=\emptyset$. Then $(L_{J_1}^{\text{reg}}\cap U^+)\cdot(L_{J_2}^{\text{reg}}\cap U^-)$ lies in a single unipotent conjugacy class. 

\begin{proof}
Denote $w_{J_2}$ to be the longest element in the parabolic subgroup $W_{J_2}$. Since $J_1\cap J_2 = \emptyset$, we have that $$\dot{w}_{J_2}(L_{J_1}^{\text{reg}}\cap U^+)\cdot(L_{J_2}^{\text{reg}}\cap U^-)\dot{w}_{J_2}^{-1}\subseteq U^+.$$
Hence, are elements in this set are unipotent.

Suppose that $x_1, x_2\in L_{J_1}^{\text{reg}}\cap U^+$ and $y_1,y_2\in L_{J_2}^{\text{reg}}\cap U^-$. Let $z_1 = x_1y_1$ and $z_2 = x_2y_2$. We want to show that $z_1$ and $z_2$ are conjugate.

Since both $x_1$ and $x_2$ are regular unipotent element of $L_{J_1}$, we can find $g\in L_{J_1}$ such that $gx_1g^{-1}=x_2$. Now
$$gx_1g^{-1}\in g(L_{J_1}\cap B^+)g^{-1}\cap (L_{J_1}\cap B^+),$$
we have that $g(L_{J_1}\cap B^+)g^{-1}=L_{J_1}\cap B^+$ since a regular unipotent element lies in a unique Borel subgroup. Hence, $g\in L_{J_1}\cap B^+$. We may write $g = tu$ where $t\in T$ and $u\in U_{J_1}^+$, the unipotent radical of $L_{J_1}\cap B^+$.

Note that $u$ can be written as a product of root group elements with respect to positive roots spanned by $J_1$, while $y_1$ can be written as a product of root group elements with respect to negative roots spanned by $J_2$. Since $J_1\cap J_2=\emptyset$, a sum of such two roots is never a root. Hence, $u$ commutes with $y_1$. Then $gz_1g^{-1} = x_2(ty_1t^{-1})$. 

Then, take $t_0\in T$ such that $\alpha(t_0)=1$ for all $\alpha\in J_1$ and $\beta(t_0) = \beta(t)^{-1}$ for all $\beta\in J_2$. Again, by the assumption that $J_1\cap J_2=\emptyset$, such a $t_0$ always exists.  Then  $(gt_0)z_1(gt_0)^{-1} = x_2y_1$. 

Finally, the same technique can be applied to $x_2y_1$, namely, there is an $h\in G$ such that $h(x_2y_1)h^{-1} = z_2$, so $z_1$ is conjugate to $z_2$.
\end{proof}

\end{lemma}

\begin{corollary} \label{cor:singleConjClass}
The cell $U^+(w_1)U^-(w_2)$ lies in a single unipotent conjugacy class depending only on $J_1=\text{supp}(w_1)$ and $J_2=\text{supp}(w_2)$, not on $w_1$ and $w_2$.

\begin{proof}
Let $J_1=\text{supp}(w_1)$ and $J_2=\text{supp}(w_2)$. The element $u_{\alpha_{i_1}}(a_1)u_{\alpha_{i_2}(a_2)}\cdots u_{i_k}(a_k)$ in $U^+(w_1)$ has image $u_\alpha(\sum_{\alpha_i=\alpha}a_i)\neq e$ under the projection of $U^+$ to the simple root group $U_\alpha^+$. Hence, $U^+(w_1)\subseteq L_{J_1}^{\text{reg}}$, and similar for $U^-(w_2)$. We conclude that $U^+(w_1)\subseteq L_{J_1}^{\text{reg}}\cap U^+$ and $U^-(w_2)\subseteq L_{J_2}^{\text{reg}}\cap U^-$.
\end{proof}

\end{corollary}

\begin{remark} \label{equivalent} \cref{lem:uniqueness} has a Lie algebra version where $\left(L_{J_1}^{\text{reg}}\cap U^+\right)\cdot \left(L_{J_2}^{\text{reg}}\cap U^-\right)$ is replaced by $\left(\mathfrak l_{J_1}^{\text{reg}}\cap\mathfrak u^+\right)\oplus\left(\mathfrak l_{J_2}^{\text{reg}}\cap\mathfrak u^-\right)$. Moreover, the nilpotent orbit so obtained corresponds to the unipotent conjugacy class in this proposition under the bijection via the exponential map. 

\end{remark}

\begin{notation}
\begin{enumerate}
    \item Denote the unipotent conjugacy class that $\left(L_{J_1}^{\text{reg}}\cap U^+\right)\cdot \left(L_{J_2}^{\text{reg}}\cap U^-\right)$ lies in by $\mathcal C_{(J_1,J_2)}$, and the respective nilpotent orbit by $\mathcal O_{(J_1,J_2)}$.
    \item Denote the unipotent conjugacy class that the regular unipotent elements of $L_J$ lie in by $\mathcal C_J$, and the respective nilpotent orbit by $\mathcal O_J$.
\end{enumerate}
\end{notation}

The rest of this paper is devoted to proving the following theorem, confirming Lusztig's conjecture.

\begin{theorem}  \label{thm:regularity}
There is a Levi subgroup $L$ such that every element in $\mathcal C_{(J_1,J_2)}\cap L$ is regular in $L$.
\end{theorem}

Here we give an overview of the proof. First, note that it is sufficient to find a Levi subalgebra $\mathfrak l$ of $\mathfrak g$ and show that $\mathcal{O}_{(J_1,J_2)}\cap\mathfrak l\subseteq\mathfrak{l}^{\text{reg}}$ since once this is completed, we may pass to the similar conclusion of unipotent conjugacy classes by the exponential map. Consider a representative $$X \defeq \sum_{\alpha\in J_1}X_\alpha+\sum_{\beta\in J_2}X_{-\beta}\in\mathcal O_{(J_1,J_2)}.$$ It is sufficient to determine the nilpotent orbit of $X$. In type $A$ and type $D$, we give an algorithm to calculate the partition of $X$ and determine the nilpotent orbit of $X$ in terms of Bala-Carter theory. In type $E$, we use a sequence of group elements to perform adjoint actions on $X$ such that we end up with a nilpotent element of the form $\sum_{\gamma\in J}X_\gamma$ for some $J\subseteq I$. In this way, we show that $\mathfrak{l}_J$ is exactly a minimal Levi subalgebra of $\mathcal O_{(J_1,J_2)}$ and $\mathcal{O}_{(J_1,J_2)}\cap\mathfrak l_J\subseteq\mathfrak{l}_J^{\text{reg}}$. We derive the result of non-simply-laced types via folding.

\begin{remark}
Unlike the cell $U^+(w_1)U^-(w_2)$, the set $\left(L_{J_1}^{\text{reg}}\cap U^+\right)\cdot \left(L_{J_2}^{\text{reg}}\cap U^-\right)$ is defined over any characteristic, and \cref{lem:uniqueness} does not depend on the characteristic. It would be an interesting question to determine whether \cref{thm:regularity} still holds in positive characteristics or not. The respective calculation will be more involved since one can no longer pass to Lie algebra.
\end{remark}

\cref{thm:regularity} allows us to give the following definition.

\begin{definition} \label{fusion}
Let $J$ be a subset of $I$ such that $L_J$ is a standard Levi subgroup such that the elements in $\mathcal C_{(J_1,J_2)}\cap L_J$ are regular in $L_J$. This gives rise to a map $$\mathcal F: \{(J_1,J_2)\mid J_1,J_2\subseteq I,J_1\cap J_2=\emptyset\}\to \{J\mid J\subseteq I\}/\sim$$ which we call the \emph{fusion map}. Here, $\sim$ is an equivalent relation on the set $\{J|J\subseteq I\}$ defined by $J\sim J^\prime$ if and only if there is $w\in W$ such that $J = w\cdot J^\prime$.
\end{definition}

\subsection{Assumption and some useful lemmas}

It is easy to show that if \cref{thm:regularity} holds for every simple factor of $G$ then it holds for $G$. It is enough to consider the Levi subgroup $L_{J_1\cup J_2}$ if $I\neq J_1\cup J_2$. In the rest of the paper, we will make the following assumption.

\begin{assumption}
$G$ is quasi-simple and $I = J_1\cup J_2$.
\end{assumption}

We prove two lemmas that will be useful in the sequel. They are stated in terms of unipotent conjugacy classes, but there are counterparts regarding nilpotent orbits.
  
  \begin{lemma}  \label{flip}
  Assume that $\mathcal C_{(J_1,J_2)} = \mathcal C_J$. Then, interchanging the roles of $J_1$ and $J_2$ does not change the corresponding unipotent conjugacy class, namely, $\mathcal C_{(J_1,J_2)} = \mathcal C_{(J_2,J_1)} = \mathcal C_J$. 
 \begin{proof}
  For a fixed pinning there is an automorphism $\phi$ of $G$ such that $\phi(u_{\alpha}(a)) = u_{-\alpha}(-a)$ for $\alpha\in I\cup -I$, $a\in\mathbb C$ and $\phi(t) = t^{-1}$ for $t\in T$. Then $$\phi\left(\prod_{\alpha\in J_1}u_{\alpha}(1)\prod_{\beta\in J_2}u_{-\beta}(1)\right)=\prod_{\alpha\in J_1}u_{-\alpha}(-1)\prod_{\beta\in J_2}u_{\beta}(-1).$$
  
  By our assumption, $\prod_{\alpha\in J_1}u_{\alpha}(1)\prod_{\beta\in J_2}u_{-\beta}(1)$ is conjugated to $\prod_{\gamma\in J}u_\gamma(1)$. Since $\phi$ is an automorphism, it preserves conjugation, so $\prod_{\alpha\in J_1}u_{-\alpha}(-1)\prod_{\beta\in J_2}u_{\beta}(-1)$ is conjugated to $\phi\left(\prod_{\gamma\in J}u(1)\right) = \prod_{\gamma\in J}u_{-\gamma}(-1)$. Two regular unipotent elements are conjugated to each other, so we know that $\prod_{\alpha\in J_1}u_{\alpha}(1)\prod_{\beta\in J_2}u_{-\beta}(1)$ is conjugated to $\prod_{\alpha\in J_1}u_{-\alpha}(-1)\prod_{\beta\in J_2}u_{\beta}(-1)$ and $\mathcal C_{(J_1,J_2)} = \mathcal C_{(J_2,J_1)} = \mathcal C_J$.

  \end{proof}
  \end{lemma}

   \begin{lemma} \label{DiagramAuto}
     Let $\sigma$ be a diagram automorphism of the Dynkin diagram. If $\mathcal C_{(J_1,J_2)} = \mathcal C_J$, then $\mathcal C_{(\sigma(J_1),\sigma(J_2))} = \mathcal C_{\sigma(J)}$.
     \begin{proof}
       The diagram automorphism $\sigma$ can be lifted to an automorphism $\dot{\sigma}:G\to G$, satisfying $\dot{\sigma}(u_\alpha(a)) = u_{\sigma(\alpha)}(a)$. Now if for $g\in G$, $$g\left(\prod_{\alpha\in J_1}u_{\alpha}(1)\prod_{\beta\in J_2}u_{-\beta}(1)\right)g^{-1} = \prod_{\gamma\in J}u_\gamma(1),$$
        then we can apply $\dot{\sigma}$ and get
        $$\dot{\sigma}(g)\left(\prod_{\alpha\in \sigma(J_1)}u_{\alpha}(1)\prod_{\beta\in \sigma(J_2)}u_{-\beta}(1)\right)\dot{\sigma}(g)^{-1} = \prod_{\gamma\in \sigma(J)}u_\gamma(1).$$
     \end{proof}
   \end{lemma}

We postpone the proof of \cref{thm:regularity} after \cref{section:weight}.

\section{The weight algorithm}\label{section:weight}
    In this section, we provide a uniform algorithm to calculate the image of the fusion map \ref{fusion} for simply-laced types. The algorithm is graph-theoretic so we will not distinguish vertices of the Dynkin diagram and roots.


First, we introduce the set-ups. 

\begin{definition} Consider a simply-laced Dynkin diagram $\Gamma$ and let $I$ be the set of vertices. Let $J_1$ and $J_2$ be two disjoint subsets of $I$. We slightly abuse notations and use $K$ to indicate the Dynkin subdiagram with vertices set $K\subseteq I$. 
\begin{enumerate}
    \item For each vertex $v\in I$ we attach a label $+$ (resp. $-$) if $v\in J_1$ (resp. $J_2$), and obtain the \emph{labeled Dynkin diagram} ($\Gamma,J_1,J_2$).
    \item A \emph{chunk} in a labeled Dynkin diagram is a connected component $C$ of $J_1$ or $J_2$, which is again a Dynkin diagram.
    \item The \emph{rank} of a chunk $C$ is the number of vertices in $C$.
    \item The \emph{type} of a chunk is its type as a Dynkin diagram.
\end{enumerate}
    
\end{definition}  

Now fix a labeled Dynkin diagram ($\Gamma,J_1,J_2$) of type $X$. For a chunk of type $Y$ and of rank $r$, we assign a number to the chunk called \emph{weight} given in the following table. Here, we omit the labels of chunks as the weights are independent of the labels.

\begin{center}
\begin{tabular}{|c|c|c|c|c|}
\hline
  $X$ & $Y$ & $r$ & Weight & Notation \\
    \hline
     $A$ & $A_n$ & $n$ & $n+1$ & NA \\
     \hline
     \multirow{3}*{$D$} &$A_n$ & $n$ & $n+1$ & NA \\
     \cline{2-5}
     &$D_n$ & $n\geq 4$ & $2n-1$ & NA \\
     \cline{2-5}
     &$\dynkin[edge length=.5cm,upside down] D{x***}$ & $3$ & $5$ & $D_3$ \\
     \hline
     \multirow{6}*{$E$}&$A_n$ & $n$ & $n+1$ & NA \\
     \cline{2-5}
    & $D_n$ & $n\geq 4$ & $2n-1$ & NA \\
     \cline{2-5}
    & $E$ & $6,7,8$ & $\infty$ & NA \\
     \cline{2-5}
   & $\dynkin[edge length=.5cm, ] D{x**x}$ & $2$ & $3$ & $A^b$\\
     \cline{2-5}
    & $\dynkin[edge length=.5cm] D{x**x*}$ & $3$ & $4$ & $A^\#$\\
     \cline{2-5}
   & $\dynkin[edge length=.5cm,upside down] E{****xxxx}$ & $4$ & $5$ & $A^\#$\\
     \hline
\end{tabular}
\end{center}

The symbol $\times$ instead of $\bullet$ here indicates that the vertices lie in the parent diagram but do not lie in the chunk itself. Note that three kinds of chunks deserve special notation. The name $D_3$ suggests that it has weight $5$ instead of $4$ as usual $A_3$ chunk. The names $A^b$ and $A^\#$ suggest some priority in the algorithm, which will be clear in the \cref{algorithm: weight_alg}

We need more terminology in the algorithm. 

\begin{definition} In the setting above, 
    \begin{enumerate}
    \item  a \emph{dominant chunk} is a chunk with the largest weight among all chunks, with the convention that $\infty$ is larger than any integer;
        \item a \emph{selection} $P$ of dominant chunks is a collection of mutually disjoint dominant chunks, i.e., the union of any two dominant chunks is not a connected Dynkin diagram as a subset of $I$;
        \item if $\Gamma$ is of type $D$ or $E$, the unique vertex that is connected to three vertices is called the \emph{branching vertex};
        \item a vertex $v\in \Gamma$ is an \emph{extra vertex} to a selection $P$ if
        \begin{itemize}
            \item $v$ is not in any chunks in $P$, and
            \item $v$ is only connected to the branching vertex that is contained in some dominant chunk of $P$.
        \end{itemize}
    \end{enumerate}
\end{definition}

We now introduce the \emph{weight algorithm}.

\begin{algorithm}
\label{algorithm: weight_alg}
Let ($\Gamma,J_1,J_2$) be a labeled Dynkin diagram.
We will do the following procedure to find one dominant selection $P$ of $\Gamma$ according to the following priority.

\begin{enumerate}[label=\textbf{Step \arabic*}.]
    \item $\mathcal{P}_1\defeq$\{the selections $P$ with maximal number of dominant chunks\}. If $\mathcal P_1$ contains only one element, then note down this element and jump to Step 7, else continue.
    \item $\mathcal{P}_2\defeq$\{the selections $P\in \mathcal{P}_1$ with maximal number of type $A$ chunks\}. If $\mathcal P_2$ contains only one element, then note down this element and jump to Step 7, else continue.
    \item $\mathcal{P}_3\defeq$\{the selections $P\in\mathcal{P}_2$ that are not adjacent to any extra vertices, where adjacent means that some dominant chunks in $P$ are connected to an extra vertex\}. If $\mathcal{P}_3=\emptyset$, then $\mathcal{P}_3\defeq\mathcal{P}_2$. If $\mathcal P_3$ contains only one element, then note down this element and jump to Step 7, else continue.
    \item $\mathcal{P}_4\defeq$\{the selections $P\in\mathcal{P}_3$ that has dominant chunks $A^\#$\}. If $\mathcal{P}_4=\emptyset$, then $\mathcal{P}_4\defeq\mathcal{P}_3$. If $\mathcal P_4$ contains only one element, then note down this element and jump to Step 7, else continue.
    \item $\mathcal{P}_5\defeq$\{the selections $P\in\mathcal{P}_4$ that does not have dominant chunks $A^b$\}. If $\mathcal{P}_5=\emptyset$, then $\mathcal{P}_5\defeq\mathcal{P}_4$. If $\mathcal P_5$ contains only one element, then note down this element and jump to Step 7, else continue.
    \item Note down any selection in $\mathcal{P}_5$.
    \item Let $P$ be the dominant selection noted down in Steps 1-6 of $\Gamma$. Then remove the label of all vertices adjacent to some chunks in $\Gamma$, call it $\Gamma^\prime$ and change the labeling of selected chunks to $J$. Go to Step 1 again for each connected component of $\Gamma^\prime$, until all the vertices are either labeled $J$ or not labeled.
\end{enumerate}

\end{algorithm}

We provide four examples covering tie-breaking cases.

\begin{example} \label{AlgorithmExample}

Consider type $D_{16}$ with $J_1=\{2,3,4,5,10,11,12,13\}$ labeled as ``$+$'' in the Dynkin diagram, $J_2=\{1,6,7,8,9,14,15,16\}$ labeled as ``$-$'' in the Dynkin diagram.
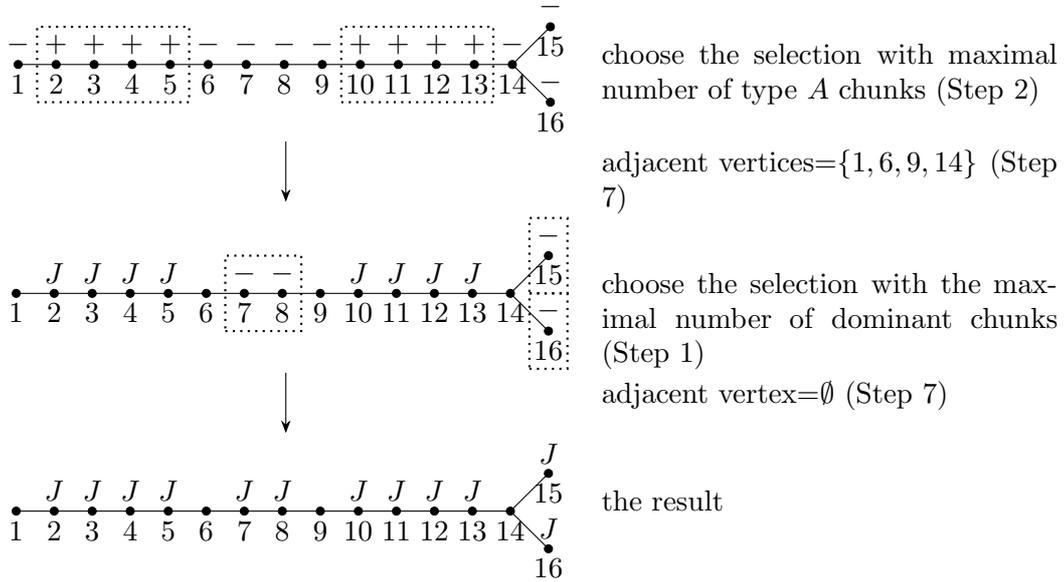
\begin{center}
         \begin{tabular}{c p{6cm}}
              \tikz[scale=0.5,decoration={markings, mark= at position 1.5 with {\arrow{stealth}}},baseline=(current bounding box.center),every circle node/.style={draw}]{
                \draw[fill=black] (-7,0) circle (3pt) node[anchor=south] {$-$} node[anchor=north] {$1$};
                \draw[fill=black] (-6,0) circle (3pt) node[anchor=south] {$+$} node[anchor=north] {$2$};
                \draw[fill=black] (-5,0) circle (3pt) node[anchor=south] {$+$} node[anchor=north] {$3$};
                \draw[fill=black] (-4,0) circle (3pt) node[anchor=south] {$+$} node[anchor=north] {$4$};
                \draw[fill=black] (-3,0) circle (3pt) node[anchor=south] {$+$} node[anchor=north] {$5$};
                \draw[fill=black] (-2,0) circle (3pt) node[anchor=south] {$-$} node[anchor=north] {$6$};
                \draw[fill=black] (-1,0) circle (3pt) node[anchor=south] {$-$} node[anchor=north] {$7$};
                \draw[fill=black] (0,0) circle (3pt) node[anchor=south] {$-$} node[anchor=north] {$8$};
                \draw[fill=black] (1,0) circle (3pt) node[anchor=south] {$-$} node[anchor=north] {$9$};
                \draw[fill=black] (2,0) circle (3pt) node[anchor=south] {$+$} node[anchor=north] {$10$};
                \draw[fill=black] (3,0) circle (3pt) node[anchor=south] {$+$} node[anchor=north] {$11$};
                \draw[fill=black] (4,0) circle (3pt) node[anchor=south] {$+$} node[anchor=north] {$12$};
                \draw[fill=black] (5,0) circle (3pt) node[anchor=south] {$+$} node[anchor=north] {$13$};
                \draw[fill=black] (6,0) circle (3pt) node[anchor=south] {$-$} node[anchor=north] {$14$};
                \draw[fill=black] (7,1) circle (3pt) node[anchor=south] {$-$} node[anchor=north] {$15$};
                \draw[fill=black] (7,-1) circle (3pt) node[anchor=south] {$-$} node[anchor=north] {$16$};
                \draw (-7,0) -- (-6,0);
                \draw (-6,0) -- (-5,0);
                \draw (-5,0) -- (-4,0);
                \draw (-4,0) -- (-3,0);
                \draw (-3,0) -- (-2,0);
                \draw (-2,0) -- (-1,0);
                \draw (-1,0) -- (0,0);
                \draw (0,0) -- (1,0);
                \draw (1,0) -- (2,0);
                \draw (2,0) -- (3,0);
                \draw (3,0) -- (4,0);
                \draw (4,0) -- (5,0);
                \draw (5,0) -- (6,0);
                \draw (6,0) -- (7,1);
                \draw (6,0) -- (7,-1);
                \draw[black,thick,dotted] ($(-6.5,-1)$)  rectangle ($(-2.5,1)$);
                \draw[black,thick,dotted] ($(1.5,-1)$)  rectangle ($(5.5,1)$);
            }
               & choose the selection with maximal number of type $A$ chunks (Step 2)\\
               
               \tikz[baseline=(current bounding box.center)]{\draw[-{Stealth[black]}] (0,0.4)   -- (0,-0.4);} & {adjacent vertices=$\{1,6,9,14\}$ (Step 7)} \\
               \tikz[scale=0.5,decoration={markings, mark= at position 1.5 with {\arrow{stealth}}},baseline=(current bounding box.center),every circle node/.style={draw}]{
                \draw[fill=black] (-7,0) circle (3pt) node[anchor=south] {} node[anchor=north] {$1$};
                \draw[fill=black] (-6,0) circle (3pt) node[anchor=south] {$J$} node[anchor=north] {$2$};
                \draw[fill=black] (-5,0) circle (3pt) node[anchor=south] {$J$} node[anchor=north] {$3$};
                \draw[fill=black] (-4,0) circle (3pt) node[anchor=south] {$J$} node[anchor=north] {$4$};
                \draw[fill=black] (-3,0) circle (3pt) node[anchor=south] {$J$} node[anchor=north] {$5$};
                \draw[fill=black] (-2,0) circle (3pt) node[anchor=south] {} node[anchor=north] {$6$};
                \draw[fill=black] (-1,0) circle (3pt) node[anchor=south] {$-$} node[anchor=north] {$7$};
                \draw[fill=black] (0,0) circle (3pt) node[anchor=south] {$-$} node[anchor=north] {$8$};
                \draw[fill=black] (1,0) circle (3pt) node[anchor=south] {} node[anchor=north] {$9$};
                \draw[fill=black] (2,0) circle (3pt) node[anchor=south] {$J$} node[anchor=north] {$10$};
                \draw[fill=black] (3,0) circle (3pt) node[anchor=south] {$J$} node[anchor=north] {$11$};
                \draw[fill=black] (4,0) circle (3pt) node[anchor=south] {$J$} node[anchor=north] {$12$};
                \draw[fill=black] (5,0) circle (3pt) node[anchor=south] {$J$} node[anchor=north] {$13$};
                \draw[fill=black] (6,0) circle (3pt) node[anchor=south] {} node[anchor=north] {$14$};
                \draw[fill=black] (7,1) circle (3pt) node[anchor=south] {$-$} node[anchor=north] {$15$};
                \draw[fill=black] (7,-1) circle (3pt) node[anchor=south] {$-$} node[anchor=north] {$16$};
                \draw (-7,0) -- (-6,0);
                \draw (-6,0) -- (-5,0);
                \draw (-5,0) -- (-4,0);
                \draw (-4,0) -- (-3,0);
                \draw (-3,0) -- (-2,0);
                \draw (-2,0) -- (-1,0);
                \draw (-1,0) -- (0,0);
                \draw (0,0) -- (1,0);
                \draw (1,0) -- (2,0);
                \draw (2,0) -- (3,0);
                \draw (3,0) -- (4,0);
                \draw (4,0) -- (5,0);
                \draw (5,0) -- (6,0);
                \draw (6,0) -- (7,1);
                \draw (6,0) -- (7,-1);
                \draw[black,thick,dotted] ($(-1.5,-1)$)  rectangle ($(0.5,1)$);
                \draw[black,thick,dotted] ($(6.5,0)$)  rectangle ($(7.5,2)$);
                \draw[black,thick,dotted] ($(6.5,-2)$)  rectangle ($(7.5,0)$);
            }

               & choose the selection with the maximal number of dominant chunks (Step 1) \\
               
               \tikz[baseline=(current bounding box.center)]{\draw[-{Stealth[black]}] (0,0.4)   -- (0,-0.4);} & adjacent vertex=$\emptyset$ (Step 7)\\
               
               \tikz[scale=0.5,decoration={markings, mark= at position 1.5 with {\arrow{stealth}}},baseline=(current bounding box.center),every circle node/.style={draw}]{
                \draw[fill=black] (-7,0) circle (3pt) node[anchor=south] {} node[anchor=north] {$1$};
                \draw[fill=black] (-6,0) circle (3pt) node[anchor=south] {$J$} node[anchor=north] {$2$};
                \draw[fill=black] (-5,0) circle (3pt) node[anchor=south] {$J$} node[anchor=north] {$3$};
                \draw[fill=black] (-4,0) circle (3pt) node[anchor=south] {$J$} node[anchor=north] {$4$};
                \draw[fill=black] (-3,0) circle (3pt) node[anchor=south] {$J$} node[anchor=north] {$5$};
                \draw[fill=black] (-2,0) circle (3pt) node[anchor=south] {} node[anchor=north] {$6$};
                \draw[fill=black] (-1,0) circle (3pt) node[anchor=south] {$J$} node[anchor=north] {$7$};
                \draw[fill=black] (0,0) circle (3pt) node[anchor=south] {$J$} node[anchor=north] {$8$};
                \draw[fill=black] (1,0) circle (3pt) node[anchor=south] {} node[anchor=north] {$9$};
                \draw[fill=black] (2,0) circle (3pt) node[anchor=south] {$J$} node[anchor=north] {$10$};
                \draw[fill=black] (3,0) circle (3pt) node[anchor=south] {$J$} node[anchor=north] {$11$};
                \draw[fill=black] (4,0) circle (3pt) node[anchor=south] {$J$} node[anchor=north] {$12$};
                \draw[fill=black] (5,0) circle (3pt) node[anchor=south] {$J$} node[anchor=north] {$13$};
                \draw[fill=black] (6,0) circle (3pt) node[anchor=south] {} node[anchor=north] {$14$};
                \draw[fill=black] (7,1) circle (3pt) node[anchor=south] {$J$} node[anchor=north] {$15$};
                \draw[fill=black] (7,-1) circle (3pt) node[anchor=south] {$J$} node[anchor=north] {$16$};
                \draw (-7,0) -- (-6,0);
                \draw (-6,0) -- (-5,0);
                \draw (-5,0) -- (-4,0);
                \draw (-4,0) -- (-3,0);
                \draw (-3,0) -- (-2,0);
                \draw (-2,0) -- (-1,0);
                \draw (-1,0) -- (0,0);
                \draw (0,0) -- (1,0);
                \draw (1,0) -- (2,0);
                \draw (2,0) -- (3,0);
                \draw (3,0) -- (4,0);
                \draw (4,0) -- (5,0);
                \draw (5,0) -- (6,0);
                \draw (6,0) -- (7,1);
                \draw (6,0) -- (7,-1);
            }
            & the result \\
            
          \end{tabular}
\end{center}
  The algorithm produces $J=\{2,3,4,5,7,8,10,11,12,13,15,16\}$.

\end{example}

\begin{example}

Consider type $D_{14}$ with $J_1=\{1,2,5,6,9,12,13\}$ labeled as ``$+$'' in the Dynkin diagram, $J_2=\{3,4,7,8,10,11,14\}$ labeled as ``$-$'' in the Dynkin diagram.

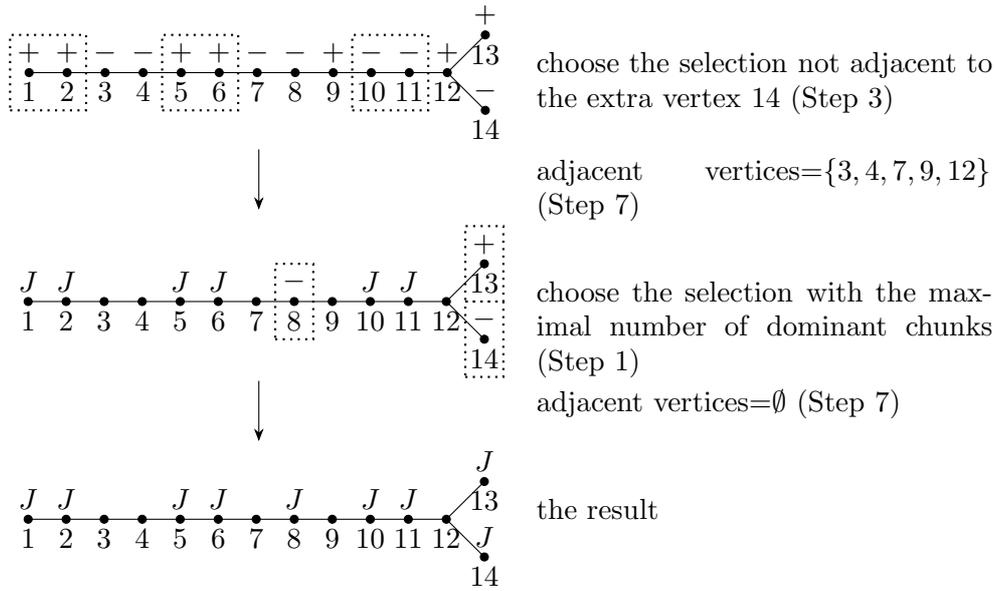
\begin{center}
    \begin{tabular}{c p{6cm}}
          \tikz[scale=0.5,decoration={markings, mark= at position 1.5 with {\arrow{stealth}}},baseline=(current bounding box.center),every circle node/.style={draw}]{
            \draw[fill=black] (-7,0) circle (3pt) node[anchor=south] {$+$} node[anchor=north] {$1$};
            \draw[fill=black] (-6,0) circle (3pt) node[anchor=south] {$+$} node[anchor=north] {$2$};
            \draw[fill=black] (-5,0) circle (3pt) node[anchor=south] {$-$} node[anchor=north] {$3$};
            \draw[fill=black] (-4,0) circle (3pt) node[anchor=south] {$-$} node[anchor=north] {$4$};
            \draw[fill=black] (-3,0) circle (3pt) node[anchor=south] {$+$} node[anchor=north] {$5$};
            \draw[fill=black] (-2,0) circle (3pt) node[anchor=south] {$+$} node[anchor=north] {$6$};
            \draw[fill=black] (-1,0) circle (3pt) node[anchor=south] {$-$} node[anchor=north] {$7$};
            \draw[fill=black] (0,0) circle (3pt) node[anchor=south] {$-$} node[anchor=north] {$8$};
            \draw[fill=black] (1,0) circle (3pt) node[anchor=south] {$+$} node[anchor=north] {$9$};
            \draw[fill=black] (2,0) circle (3pt) node[anchor=south] {$-$} node[anchor=north] {$10$};
            \draw[fill=black] (3,0) circle (3pt) node[anchor=south] {$-$} node[anchor=north] {$11$};
            \draw[fill=black] (4,0) circle (3pt) node[anchor=south] {$+$} node[anchor=north] {$12$};
            \draw[fill=black] (5,1) circle (3pt) node[anchor=south] {$+$} node[anchor=north] {$13$};
            \draw[fill=black] (5,-1) circle (3pt) node[anchor=south] {$-$} node[anchor=north] {$14$};
            \draw (-7,0) -- (-6,0);
            \draw (-6,0) -- (-5,0);
            \draw (-5,0) -- (-4,0);
            \draw (-4,0) -- (-3,0);
            \draw (-3,0) -- (-2,0);
            \draw (-2,0) -- (-1,0);
            \draw (-1,0) -- (0,0);
            \draw (0,0) -- (1,0);
            \draw (1,0) -- (2,0);
            \draw (2,0) -- (3,0);
            \draw (3,0) -- (4,0);
            \draw (4,0) -- (5,1);
            \draw (4,0) -- (5,-1);
            \draw[black,thick,dotted] ($(-7.5,-1)$)  rectangle ($(-5.5,1)$);
            \draw[black,thick,dotted] ($(-3.5,-1)$)  rectangle ($(-1.5,1)$);
            \draw[black,thick,dotted] ($(1.5,-1)$)  rectangle ($(3.5,1)$);
        }
           & choose the selection not adjacent to the extra vertex $14$ (Step 3)\\
           
           \tikz[baseline=(current bounding box.center)]{\draw[-{Stealth[black]}] (0,0.4)   -- (0,-0.4);} & {adjacent vertices=$\{3,4,7,9,12\}$ (Step 7)} \\
           
           \tikz[scale=0.5,decoration={markings, mark= at position 1.5 with {\arrow{stealth}}},baseline=(current bounding box.center),every circle node/.style={draw}]{
            \draw[fill=black] (-7,0) circle (3pt) node[anchor=south] {$J$} node[anchor=north] {$1$};
            \draw[fill=black] (-6,0) circle (3pt) node[anchor=south] {$J$} node[anchor=north] {$2$};
            \draw[fill=black] (-5,0) circle (3pt) node[anchor=south] {} node[anchor=north] {$3$};
            \draw[fill=black] (-4,0) circle (3pt) node[anchor=south] {} node[anchor=north] {$4$};
            \draw[fill=black] (-3,0) circle (3pt) node[anchor=south] {$J$} node[anchor=north] {$5$};
            \draw[fill=black] (-2,0) circle (3pt) node[anchor=south] {$J$} node[anchor=north] {$6$};
            \draw[fill=black] (-1,0) circle (3pt) node[anchor=south] {} node[anchor=north] {$7$};
            \draw[fill=black] (0,0) circle (3pt) node[anchor=south] {$-$} node[anchor=north] {$8$};
            \draw[fill=black] (1,0) circle (3pt) node[anchor=south] {} node[anchor=north] {$9$};
            \draw[fill=black] (2,0) circle (3pt) node[anchor=south] {$J$} node[anchor=north] {$10$};
            \draw[fill=black] (3,0) circle (3pt) node[anchor=south] {$J$} node[anchor=north] {$11$};
            \draw[fill=black] (4,0) circle (3pt) node[anchor=south] {} node[anchor=north] {$12$};
            \draw[fill=black] (5,1) circle (3pt) node[anchor=south] {$+$} node[anchor=north] {$13$};
            \draw[fill=black] (5,-1) circle (3pt) node[anchor=south] {$-$} node[anchor=north] {$14$};
            \draw (-7,0) -- (-6,0);
            \draw (-6,0) -- (-5,0);
            \draw (-5,0) -- (-4,0);
            \draw (-4,0) -- (-3,0);
            \draw (-3,0) -- (-2,0);
            \draw (-2,0) -- (-1,0);
            \draw (-1,0) -- (0,0);
            \draw (0,0) -- (1,0);
            \draw (1,0) -- (2,0);
            \draw (2,0) -- (3,0);
            \draw (3,0) -- (4,0);
            \draw (4,0) -- (5,1);
            \draw (4,0) -- (5,-1);
            \draw[black,thick,dotted] ($(-0.5,-1)$)  rectangle ($(0.5,1)$);
            \draw[black,thick,dotted] ($(4.5,0)$)  rectangle ($(5.5,2)$);
            \draw[black,thick,dotted] ($(4.5,-2)$)  rectangle ($(5.5,0)$);
        }
           & choose the selection with the maximal number of dominant chunks (Step 1) \\
           
           \tikz[baseline=(current bounding box.center)]{\draw[-{Stealth[black]}] (0,0.4)   -- (0,-0.4);} & adjacent vertices=$\emptyset$ (Step 7)\\
           
           \tikz[scale=0.5,decoration={markings, mark= at position 1.5 with {\arrow{stealth}}},baseline=(current bounding box.center),every circle node/.style={draw}]{
            \draw[fill=black] (-7,0) circle (3pt) node[anchor=south] {$J$} node[anchor=north] {$1$};
            \draw[fill=black] (-6,0) circle (3pt) node[anchor=south] {$J$} node[anchor=north] {$2$};
            \draw[fill=black] (-5,0) circle (3pt) node[anchor=south] {} node[anchor=north] {$3$};
            \draw[fill=black] (-4,0) circle (3pt) node[anchor=south] {} node[anchor=north] {$4$};
            \draw[fill=black] (-3,0) circle (3pt) node[anchor=south] {$J$} node[anchor=north] {$5$};
            \draw[fill=black] (-2,0) circle (3pt) node[anchor=south] {$J$} node[anchor=north] {$6$};
            \draw[fill=black] (-1,0) circle (3pt) node[anchor=south] {} node[anchor=north] {$7$};
            \draw[fill=black] (0,0) circle (3pt) node[anchor=south] {$J$} node[anchor=north] {$8$};
            \draw[fill=black] (1,0) circle (3pt) node[anchor=south] {} node[anchor=north] {$9$};
            \draw[fill=black] (2,0) circle (3pt) node[anchor=south] {$J$} node[anchor=north] {$10$};
            \draw[fill=black] (3,0) circle (3pt) node[anchor=south] {$J$} node[anchor=north] {$11$};
            \draw[fill=black] (4,0) circle (3pt) node[anchor=south] {} node[anchor=north] {$12$};
            \draw[fill=black] (5,1) circle (3pt) node[anchor=south] {$J$} node[anchor=north] {$13$};
            \draw[fill=black] (5,-1) circle (3pt) node[anchor=south] {$J$} node[anchor=north] {$14$};
            \draw (-7,0) -- (-6,0);
            \draw (-6,0) -- (-5,0);
            \draw (-5,0) -- (-4,0);
            \draw (-4,0) -- (-3,0);
            \draw (-3,0) -- (-2,0);
            \draw (-2,0) -- (-1,0);
            \draw (-1,0) -- (0,0);
            \draw (0,0) -- (1,0);
            \draw (1,0) -- (2,0);
            \draw (2,0) -- (3,0);
            \draw (3,0) -- (4,0);
            \draw (4,0) -- (5,1);
            \draw (4,0) -- (5,-1);
        }
        & the result  \\
      \end{tabular}
    \end{center}

  The algorithm produces $J=\{1,2,5,6,8,10,11,13,14\}$.

\end{example}

\begin{example}

Consider type $E_8$ with $J_1=\{1,3,5,6\}$ labeled as ``$+$'' in the Dynkin diagram, $J_2=\{2,4,7,8\}$ labeled as ``$-$'' in the Dynkin diagram.

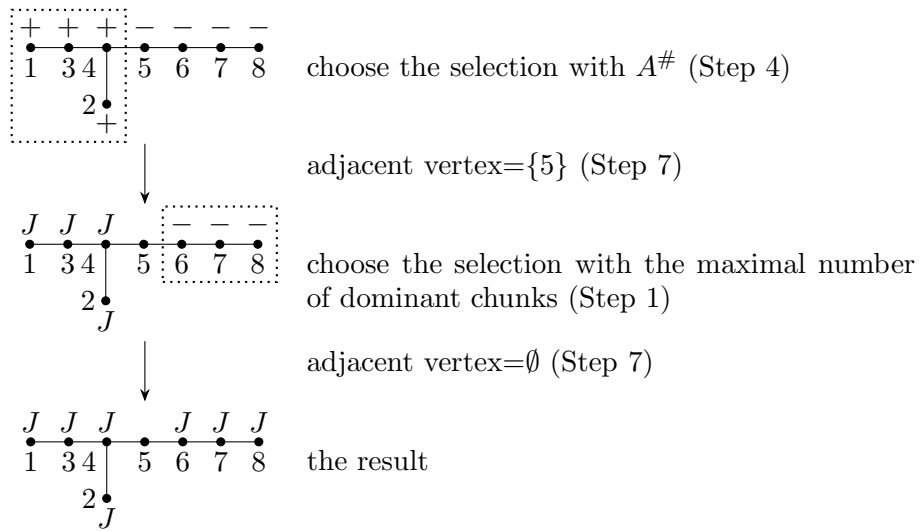
\begin{center}
        \begin{tabular}{c p{8cm}}
             \tikz[scale=0.5,decoration={markings, mark= at position 1.5 with {\arrow{stealth}}},baseline=(current bounding box.center),every circle node/.style={draw}]{
                \draw[fill=black] (-7,0) circle (3pt) node[anchor=south] {$+$} node[anchor=north] {$1$};
                \draw[fill=black] (-6,0) circle (3pt) node[anchor=south] {$+$} node[anchor=north] {$3$};
                \draw[fill=black] (-5,0) circle (3pt) node[anchor=south] {$+$} node[anchor=north east] {$4$};
                \draw[fill=black] (-4,0) circle (3pt) node[anchor=south] {$-$} node[anchor=north] {$5$};
                \draw[fill=black] (-3,0) circle (3pt) node[anchor=south] {$-$} node[anchor=north] {$6$};
                \draw[fill=black] (-2,0) circle (3pt) node[anchor=south] {$-$} node[anchor=north] {$7$};
                \draw[fill=black] (-1,0) circle (3pt) node[anchor=south] {$-$} node[anchor=north] {$8$};
                \draw[fill=black] (-5,-1.5) circle (3pt) node[anchor=north] {$+$} node[anchor=east] {$2$};
                \draw (-7,0) -- (-6,0);
                \draw (-6,0) -- (-5,0);
                \draw (-5,0) -- (-4,0);
                \draw (-4,0) -- (-3,0);
                \draw (-3,0) -- (-2,0);
                \draw (-2,0) -- (-1,0);
                \draw (-5,0) -- (-5,-1.5);
                \draw[black,thick,dotted] ($(-7.5,-2.5)$)  rectangle ($(-4.5,1)$);
            }
               & choose the selection with $A^\#$ (Step 4) \\
               
               \tikz[baseline=(current bounding box.center)]{\draw[-{Stealth[black]}] (0,0.4)   -- (0,-0.4);} & {adjacent vertex=$\{5\}$ (Step 7)} \\
               
               \tikz[scale=0.5,decoration={markings, mark= at position 1.5 with {\arrow{stealth}}},baseline=(current bounding box.center),every circle node/.style={draw}]{
                \draw[fill=black] (-7,0) circle (3pt) node[anchor=south] {$J$} node[anchor=north] {$1$};
                \draw[fill=black] (-6,0) circle (3pt) node[anchor=south] {$J$} node[anchor=north] {$3$};
                \draw[fill=black] (-5,0) circle (3pt) node[anchor=south] {$J$} node[anchor=north east] {$4$};
                \draw[fill=black] (-4,0) circle (3pt) node[anchor=south] {} node[anchor=north] {$5$};
                \draw[fill=black] (-3,0) circle (3pt) node[anchor=south] {$-$} node[anchor=north] {$6$};
                \draw[fill=black] (-2,0) circle (3pt) node[anchor=south] {$-$} node[anchor=north] {$7$};
                \draw[fill=black] (-1,0) circle (3pt) node[anchor=south] {$-$} node[anchor=north] {$8$};
                \draw[fill=black] (-5,-1.5) circle (3pt) node[anchor=north] {$J$} node[anchor=east] {$2$};
                \draw (-7,0) -- (-6,0);
                \draw (-6,0) -- (-5,0);
                \draw (-5,0) -- (-4,0);
                \draw (-4,0) -- (-3,0);
                \draw (-3,0) -- (-2,0);
                \draw (-2,0) -- (-1,0);
                \draw (-5,0) -- (-5,-1.5);
                \draw[black,thick,dotted] ($(-3.5,-1)$)  rectangle ($(-0.5,1)$);
            }
               & choose the selection with the maximal number of dominant chunks (Step 1)\\
               
               \tikz[baseline=(current bounding box.center)]{\draw[-{Stealth[black]}] (0,0.4)   -- (0,-0.4);} & adjacent vertex=$\emptyset$ (Step 7) \\

                \tikz[scale=0.5,decoration={markings, mark= at position 1.5 with {\arrow{stealth}}},baseline=(current bounding box.center),every circle node/.style={draw}]{
                \draw[fill=black] (-7,0) circle (3pt) node[anchor=south] {$J$} node[anchor=north] {$1$};
                \draw[fill=black] (-6,0) circle (3pt) node[anchor=south] {$J$} node[anchor=north] {$3$};
                \draw[fill=black] (-5,0) circle (3pt) node[anchor=south] {$J$} node[anchor=north east] {$4$};
                \draw[fill=black] (-4,0) circle (3pt) node[anchor=south] {} node[anchor=north] {$5$};
                \draw[fill=black] (-3,0) circle (3pt) node[anchor=south] {$J$} node[anchor=north] {$6$};
                \draw[fill=black] (-2,0) circle (3pt) node[anchor=south] {$J$} node[anchor=north] {$7$};
                \draw[fill=black] (-1,0) circle (3pt) node[anchor=south] {$J$} node[anchor=north] {$8$};
                \draw[fill=black] (-5,-1.5) circle (3pt) node[anchor=north] {$J$} node[anchor=east] {$2$};
                \draw (-7,0) -- (-6,0);
                \draw (-6,0) -- (-5,0);
                \draw (-5,0) -- (-4,0);
                \draw (-4,0) -- (-3,0);
                \draw (-3,0) -- (-2,0);
                \draw (-2,0) -- (-1,0);
                \draw (-5,0) -- (-5,-1.5);
                }
                
                & the result\\

          \end{tabular}
    \end{center}
    The algorithm produces $J=\{1,2,3,4,6,7,8\}$.

\end{example}

\begin{example}
   
  Consider type $E_8$ with $J_1=\{1,2,3,4\}$ labeled as ``$+$'' in the Dynkin diagram, $J_2=\{5,6,7,8\}$ labeled as ``$-$'' in the Dynkin diagram.
    \begin{center}
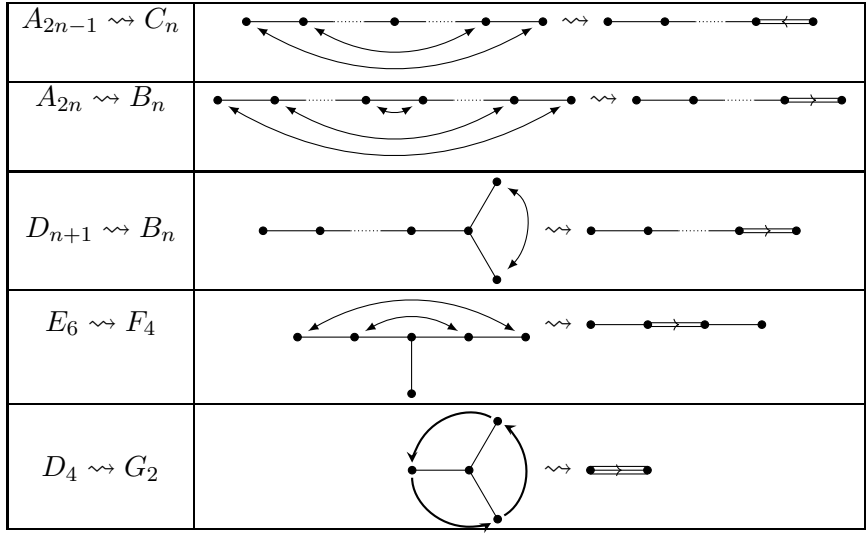
        \begin{tabular}{c p{8cm}}
         \tikz[scale=0.5,decoration={markings, mark= at position 1.5 with {\arrow{stealth}}},baseline=(current bounding box.center),every circle node/.style={draw}]{
            \draw[fill=black] (-7,0) circle (3pt) node[anchor=south] {$+$} node[anchor=north] {$1$};
            \draw[fill=black] (-6,0) circle (3pt) node[anchor=south] {$+$} node[anchor=north] {$3$};
            \draw[fill=black] (-5,0) circle (3pt) node[anchor=south] {$-$} node[anchor=north west] {$4$};
            \draw[fill=black] (-4,0) circle (3pt) node[anchor=south] {$+$} node[anchor=north] {$5$};
            \draw[fill=black] (-3,0) circle (3pt) node[anchor=south] {$+$} node[anchor=north] {$6$};
            \draw[fill=black] (-2,0) circle (3pt) node[anchor=south] {$-$} node[anchor=north] {$7$};
            \draw[fill=black] (-1,0) circle (3pt) node[anchor=south] {$-$} node[anchor=north] {$8$};
            \draw[fill=black] (-5,-1) circle (3pt) node[anchor=north] {$-$} node[anchor=east] {$2$};
            \draw (-7,0) -- (-6,0);
            \draw (-6,0) -- (-5,0);
            \draw (-5,0) -- (-4,0);
            \draw (-4,0) -- (-3,0);
            \draw (-3,0) -- (-2,0);
            \draw (-2,0) -- (-1,0);
            \draw (-5,0) -- (-5,-1);
            \draw[black,thick,dotted] ($(-7.5,-1)$)  rectangle ($(-5.5,1)$);
            \draw[black,thick,dotted] ($(-4.3,-1)$)  rectangle ($(-2.5,1)$);
        }
           & choose the selection without $A^b$ (Step 5) \\
           
           \tikz[baseline=(current bounding box.center)]{\draw[-{Stealth[black]}] (0,0.4)   -- (0,-0.4);} & {adjacent vertices=$\{4,7\}$ (Step 1) } \\
           
           \tikz[scale=0.5,decoration={markings, mark= at position 1.5 with {\arrow{stealth}}},baseline=(current bounding box.center),every circle node/.style={draw}]{
            \draw[fill=black] (-7,0) circle (3pt) node[anchor=south] {$J$} node[anchor=north] {$1$};
            \draw[fill=black] (-6,0) circle (3pt) node[anchor=south] {$J$} node[anchor=north] {$3$};
            \draw[fill=black] (-5,0) circle (3pt) node[anchor=south] {} node[anchor=north west] {$4$};
            \draw[fill=black] (-4,0) circle (3pt) node[anchor=south] {$J$} node[anchor=north] {$5$};
            \draw[fill=black] (-3,0) circle (3pt) node[anchor=south] {$J$} node[anchor=north] {$6$};
            \draw[fill=black] (-2,0) circle (3pt) node[anchor=south] {} node[anchor=north] {$7$};
            \draw[fill=black] (-1,0) circle (3pt) node[anchor=south] {$-$} node[anchor=north] {$8$};
            \draw[fill=black] (-5,-1) circle (3pt) node[anchor=north] {$-$} node[anchor=east] {$2$};
            \draw (-7,0) -- (-6,0);
            \draw (-6,0) -- (-5,0);
            \draw (-5,0) -- (-4,0);
            \draw (-4,0) -- (-3,0);
            \draw (-3,0) -- (-2,0);
            \draw (-2,0) -- (-1,0);
            \draw (-5,0) -- (-5,-1);
            \draw[black,thick,dotted] ($(-5.7,-2)$)  rectangle ($(-4.5,-0.5)$);
            \draw[black,thick,dotted] ($(-1.5,-1)$)  rectangle ($(-0.5,1)$);
        }
           & choose the selection with maximal number of dominant chunks (Step 1) \\
           
           \tikz[baseline=(current bounding box.center)]{\draw[-{Stealth[black]}] (0,0.4)   -- (0,-0.4);} & adjacent vertex=$\emptyset$ (Step 7) \\
    
            \tikz[scale=0.5,decoration={markings, mark= at position 1.5 with {\arrow{stealth}}},baseline=(current bounding box.center),every circle node/.style={draw}]{
            \draw[fill=black] (-7,0) circle (3pt) node[anchor=south] {$J$} node[anchor=north] {$1$};
            \draw[fill=black] (-6,0) circle (3pt) node[anchor=south] {$J$} node[anchor=north] {$3$};
            \draw[fill=black] (-5,0) circle (3pt) node[anchor=south] {} node[anchor=north west] {$4$};
            \draw[fill=black] (-4,0) circle (3pt) node[anchor=south] {$J$} node[anchor=north] {$5$};
            \draw[fill=black] (-3,0) circle (3pt) node[anchor=south] {$J$} node[anchor=north] {$6$};
            \draw[fill=black] (-2,0) circle (3pt) node[anchor=south] {} node[anchor=north] {$7$};
            \draw[fill=black] (-1,0) circle (3pt) node[anchor=south] {$J$} node[anchor=north] {$8$};
            \draw[fill=black] (-5,-1) circle (3pt) node[anchor=east] {$J$} node[anchor=north] {$2$};
            \draw (-7,0) -- (-6,0);
            \draw (-6,0) -- (-5,0);
            \draw (-5,0) -- (-4,0);
            \draw (-4,0) -- (-3,0);
            \draw (-3,0) -- (-2,0);
            \draw (-2,0) -- (-1,0);
            \draw (-5,0) -- (-5,-1);
        } & the result \\

        \end{tabular}
    \end{center}
    The algorithm produces $J=\{1,2,3,5,6,8\}$.

\end{example}

\begin{theorem} \label{theorem:weightthm}
    Given a labeled Dynkin diagram $(\Gamma,J_1,J_2)$ of simply-laced types, \cref{algorithm: weight_alg} gives a representative in the equivalence class $\mathcal F(J_1,J_2)$.
\end{theorem}

  We will prove this theorem for type $A$ and $D$ in \cref{section:weightproofAD} and for type $E$ in \cref{section:weightproofE}.

\section{Folding} \label{section:folding}
In this section, we will establish \cref{thm:regularity} for non-simply-laced types from simply-laced types via folding. It is worth mentioning that our approach in \cref{section:AD} for types $A$ and $D$ works for all classical types. What we do in \cref{section:E} for type $E$ applies to all exceptional types. However, the folding method provides an elegant and uniform approach to all non-simply-laced types from simple-laced ones.

We assume that \cref{thm:regularity} and \cref{theorem:weightthm} hold for all groups of simply-laced types.
    
We will use the following setting. For details, see \cite[1.5-1.6]{L94}. Let $\tilde G$ be the simply connected covering of the derived subgroup of $G$. There is a semi-simple, simply connected and simply-laced group $\dot G$, a pinning $(\dot T,\dot{B}^+,\dot{B}^-,\dot{u}_{\alpha_i},\dot{u}_{-\alpha_i};\alpha_i\in\dot I)$ on $\dot G$ and an automorphism $\sigma:\dot G\to\dot G$ compatible with the innings such that there is an isomorphism $\dot G^{\sigma}\to\tilde G$, where  $\dot{G}^\sigma$ is the group of $\sigma$-fixed points in $\dot G$. Moreover, such isomorphism is compatible with pinnings on both sides, in particular, the subgroup $\dot G^\sigma$ has a set of simple roots $I = \dot I/\sim_\sigma$, and its Weyl group $W$ is isomorphic to $\dot{W}^\sigma$, the group of $\sigma$-fixed points in $\dot W$.

   \begin{notation}
    We will use the dot notation for everything concerned with $\dot G$, for example, $\dot J$ for a subset of $\dot I$, $\dot{L}_{\dot J}$ for a Levi subgroup of $\dot G$ and $\dot{\mathcal C}_{\dot J}$ for the regular unipotent conjugacy class containing the regular elements in $\dot{L}_{\dot J}$. 
   \end{notation}

   In this way, we can realize any non-simply-laced group as a subgroup of a simply-laced group. The automorphism $\sigma$ can be induced from the diagram automorphism of the Dynkin diagrams. More precisely, given a diagram automorphism of the Dynkin diagram of a simply-laced group $\dot G$, we can lift it to be an automorphism $\sigma$ of $\dot G$, and the $\sigma$-fixed point subgroup will have the respective non-simply-laced Dynkin diagram.

   \begin{example} \label{folding}
      We may assume that both $G$ and $\dot G$ are quasi-simple, then by the classification of connected Dynkin diagrams and the respective diagram automorphisms, we can give a list of all possible foldings as follows.
   
   \begin{center}
   \begin{tabular}{|c|c|}
   \hline
    $A_{2n-1}\rightsquigarrow C_n$ & $\dynkin[edge length=.75cm,
involutions={15;24}]{A}{**.*.**} \rightsquigarrow \dynkin[%
edge length=.75cm,]{C}{**.**}$ \\ \hline

 $A_{2n}\rightsquigarrow B_n$ & $\dynkin[edge length=.75cm,
involutions={16;25;34}]{A}{**.**.**} \rightsquigarrow \dynkin[%
edge length=.75cm,]{B}{**.**}$ \\ \hline
       
 $D_{n+1}\rightsquigarrow B_n$ & $\dynkin[edge length=.75cm,
involutions={[out=60,relative]56}]{D}{**.****} \rightsquigarrow \dynkin[%
edge length=.75cm]{B}{**.**}$ \\ \hline

 $E_6\rightsquigarrow F_4$ & $\dynkin[edge length=.75cm, upside down,
involutions={16;35}]E6 \rightsquigarrow \dynkin[%
edge length=.75cm]F4$ \\ \hline

$D_4\rightsquigarrow G_2$ & $\tikzset{/Dynkin diagram/fold style/.style={-stealth,thick,
shorten <=1mm,shorten >=1mm,}}
\dynkin[ply=3,edge length=.75cm]D4
 \rightsquigarrow \dynkin[%
edge length=.75cm]G2$
\\ \hline
        
   \end{tabular}
    \end{center}

   \end{example}

   Denote $\pi: \dot I\to I$ to be the natural projection.
   
   \begin{lemma}
    Assume that $\sigma$ is one of the diagram automorphisms given in Example \ref{folding}. If two $\sigma$-stable subsets $\dot J$ and $\dot K$ of $\dot I$ are $\dot W$-conjugate, then they are $\dot{W}^\sigma$-conjugate, in other words, $J = \pi(\dot J)$ and $K = \pi(\dot K)$ are $W$-conjugate.

    \begin{proof}
    We prove this lemma using case-by-case analysis.

    For the first folding $A_{2n-1}\to C_n$ in Example \ref{folding}, since $\dot J$ and $\dot K$ are conjugate, the root subsystem generated by them should be of the same type. Since $\sigma(\dot J)=\dot J$ and $\sigma(\dot K)=\dot K$, they are both of the type $2A_{l_1}\times 2A_{l_2}\times\cdots 2A_{l_k}\times A_{2m-1}$, and then $\pi(J)$ and $\pi(K)$ are both of the type $A_{l_1}\times A_{l_2}\times\cdots A_{l_k}\times C_m$. By the classification of conjugacy classes of subsets of roots \cite[2.3]{GP}, they are conjugate to each other under the type $C_n$ Weyl group action. The analysis for foldings $A_{2n}\to B_n$ and $D_{n+1}\to B_n$ is similar.

    For the folding $E_6\to F_4$, note that under the assumption, $J$ and $K$ would be different only when they are both of the types $A_1$, $A_1\times A_1$ or $A_1\times A_1\times A_1$. In any situation, $\pi(J)$ and $\pi(K)$ are conjugate to each other under the type $F_4$ Weyl group action. For the folding $D_4\to G_2$, $J$ and $K$ cannot be conjugate to each other and meanwhile $\sigma$-stable unless $J = K$, and the assertion is trivial.
    \end{proof}
   \end{lemma}
   
   \begin{lemma} Assume that $\sigma$ is one of the diagram automorphisms given in Example \ref{folding}.
   For $\dot J\subseteq\dot I$ with $\sigma(\dot J) = \dot J$, $\dot{\mathcal C}_{\dot J}\cap\dot{G}^\sigma$ is contained in a single $\dot{G}^\sigma$-unipotent conjugacy class.

   \begin{proof}
   We first deal with the case where $\dot J = \dot I$ and $\dot{\mathcal C}_{\dot I}$ be the regular unipotent conjugacy class of $\dot G$. Let $\dot B$ be a Borel subgroup of $\dot G$. Then $\dot{B}^\sigma = \dot{G}^\sigma\cap\dot{B}$ is a Borel subgroup of $\dot{G}^\sigma$ and the natural morphism $\iota:\dot{G}^\sigma/\dot{B}^\sigma\to \dot G/\dot B$ is an injection. Viewing both sets as the sets of Borels in respective groups, we have $B^\prime\subseteq \iota(B^\prime)$.

    For $u\in\dot{\mathcal{C}}_{\dot I}\cap\dot{G}^\sigma$, there is a Borel subgroup $B^\prime$ of $\dot{G}^\sigma$ such that $u\in B^\prime\subseteq\iota(B^\prime)$. Since $u$ is a regular unipotent element in $\dot G$, it lies in a unique Borel subgroup of $\dot G$. Since $\iota$ is injective, $u$ lies in a unique Borel subgroup of $\dot G^\sigma$. Hence, $u$ is a regular element in $\dot{G}^\sigma$. Note that here we do not need to assume that $\sigma$ is one of the diagram automorphisms in \cref{folding}.

    Now we deal with the general case. Take $u\in\dot{\mathcal C}_{\dot J}\cap\dot{G}^\sigma$, and consider the $\dot{G}^\sigma$-conjugacy class $\mathcal C(u)$ of $u$. Let $L_K$ be a standard and minimal Levi subgroup of $\mathcal C(u)$ in $\dot G^\sigma$. Then $u$ is $\dot{G}^\sigma$-conjugate to an element $u'$ in $L_K$. By taking  $u$ to be $u'$, we may assume that $u\in L_K$. Then by definition of $L_K$, $u$ is distinguished in it, i.e., $u$ does not lie in any proper Levi subgroup of $L_K$. Now consider $u$ as a unipotent element of $\dot G$, $u$ is distinguished in $\dot{L}_{\pi^{-1}(K)}$.  Since $u\in\mathcal C_{\dot J}$, we know that $\dot L_{\dot J}$ is $\dot G$-conjugate to $\dot L_{\pi^{-1}(K)}$ and $u$ is a regular unipotent element in $\dot{L}_{\pi^{-1}(K)}$. Also, $\dot{L}_{\pi^{-1}(K)}^{\sigma} = L_K$, so $u$ is regular unipotent in $L_K$. But then regular elements in $L_K$ are $L_K$-conjugate. We have $\dot{\mathcal C}_{\dot J}\cap\dot{G}^\sigma$ contained in a single $\dot G^\sigma$-conjugacy class.  
   \end{proof}
   
   \end{lemma}

   \begin{corollary} \label{rediscover}
     Assume that $\dot{\mathcal C}_{(\pi^{-1}(J_1),\pi^{-1}(J_2))} = \dot{\mathcal C}_{\dot J}$ as conjugacy classes in $\dot G$ where $\sigma(\dot J)=\dot J$. Then $\mathcal C_{(J_1,J_2)} = \mathcal C_{\pi(\dot J)}$ as conjugacy classes in $\dot{G}^\sigma$.

     \begin{proof}
         Note that both $\mathcal C_{(J_1,J_2)}$ and $\mathcal C_{\pi(\dot J)}$ lie in $\mathcal C_{(\pi^{-1}(J_1),\pi^{-1}(J_2))}\cap\dot G^\sigma = \dot{\mathcal C}_{\dot J}\cap\dot G^\sigma$. Now apply the above lemma. 
     \end{proof}
   \end{corollary}

   Now we show that \cref{thm:regularity} holds when $G$ is non-simply-laced with the assumption that \cref{thm:regularity} and \cref{theorem:weightthm} hold when $G$ is simply-laced.

   \begin{proof}
   For each quasi-simple and non-simply-laced group $G$, we may consider its simply-laced cover $\dot G$. The notations are as before. Let $\dot J_1 = \pi^{-1}(J_1)$ and $\dot J_2 = \pi^{-1}(J_2)$. \cref{algorithm: weight_alg} for types $A$, $D$ and $E$ suggests that when $\dot{J}_1$ and $\dot{J}_2$ are $\sigma$-stable, there is always a resulting $\dot J$ that is $\sigma$-stable. Now the theorem follows from \cref{rediscover}.
   \end{proof}

    This result provides a way to calculate the fusion map in non-simply-laced types from the result of simply-laced types.
  
   \begin{example}
   We provide some examples of calculating the fusion map via folding.
   
       \begin{enumerate}
       \item $$
\dynkin[%
edge length=.6cm,
labels*={J_1,J_1,J_2}]C3 \xrightarrow{\text{unfold}}\dynkin[%
edge length=.6cm,
labels*={J_1,J_1,J_2,J_1,J_1}]A5\xrightarrow{\text{fusion}}\dynkin[%
edge length=.6cm,
labels*={J,J,,J,J}]A5\xrightarrow{\text{fold back}}\dynkin[%
edge length=.6cm,
labels*={J,J}]C3.$$

\item $$
\dynkin[%
edge length=.6cm,
labels*={J_1,J_1,J_2}]B3 \xrightarrow{\text{unfold}}\dynkin[%
edge length=.6cm,
labels*={J_1,J_1,J_2,J_2,J_1,J_1}]A6\xrightarrow{\text{fusion}}\dynkin[%
edge length=.6cm,
labels*={J,J,,,J,J}]A6\xrightarrow{\text{fold back}}\dynkin[%
edge length=.6cm,
labels*={J,J}]B3.$$

           \item $$
\dynkin[%
edge length=.6cm,
labels*={J_1,J_1,J_2}]B3 \xrightarrow{\text{unfold}}\dynkin[%
edge length=.6cm,
labels*={J_1,J_1,J_2,J_2}]D4\xrightarrow{\text{fusion}}\dynkin[%
edge length=.6cm,
labels*={J,J}]D4\xrightarrow{\text{fold back}}\dynkin[%
edge length=.6cm,
labels*={J,J}]B3.$$

\item $$
\dynkin[%
edge length=.6cm,
labels*={J_1,J_2,J_1,J_2}]F4 \xrightarrow{\text{unfold}}\dynkin[%
edge length=.6cm,
labels*={J_2,J_1,J_1,J_2,J_1,J_2},upside down]E6\xrightarrow{\text{fusion}}\dynkin[%
edge length=.6cm,
labels*={,J,J,,J},upside down]E6\xrightarrow{\text{fold back}}\dynkin[%
edge length=.6cm,
labels*={J,,J}]F4.$$

\item $$
\dynkin[%
edge length=.6cm,
labels*={J_2,J_1}]G2 \xrightarrow{\text{unfold}}\dynkin[%
edge length=.6cm,
labels*={J_1,J_2,J_1,J_1}]D4\xrightarrow{\text{fusion}}\dynkin[%
edge length=.6cm,
labels*={J,,J,J}]D4\xrightarrow{\text{fold back}}\dynkin[%
edge length=.6cm,
labels*={,J}]G2.$$
       \end{enumerate}
   \end{example}

\section{Type $A$ and Type $D$}\label{section:AD}
In this section, we prove \cref{thm:regularity} and \cref{theorem:weightthm} for type $A$ and $D$.

In classical types, except for the very even nilpotent orbits in type $D$, each nilpotent orbit is determined by its partition. Now we fix a choice of root vectors $\{X_\alpha\mid \alpha\in\Phi\}$ and then consider the element $X = \sum_{\alpha\in J_1}X_\alpha+\sum_{\beta\in J_2}X_{-\beta}\in\mathcal O_{(J_1,J_2)}$, which is itself a matrix. For type $A$ and type $D$, we introduce an algorithm to calculate the partition of $X$. This algorithm suggests certain rules that those partitions will obey for each type, and also indicate a necessary condition for very even partitions to occur. Using this information, we can construct a nilpotent element of the form $\sum_{\alpha\in J}X_\alpha$ for some $J\subseteq I$ which also lies in $\mathcal O_{(J_1,J_2)}$ and is regular in $\mathfrak l_J$. This will complete the proof of \cref{thm:regularity} for classical types.

In \cref{section:AD} and \cref{section:E}, we follow Bourbaki's convention for root numbering in the Dynkin diagrams in \cite{Bour}.

\subsection{Classification by partitions} \label{partition}
It is known that for classical types, nilpotent orbits can be classified by partitions with certain properties and some slight modification. Since we only need to deal with the simply-laced types, we review this theory with a focus on type $A$ and type $D$, following \cite{CM}.

\begin{itemize}

\item The nilpotent orbits in the Lie algebra of type $A_n$ are in correspondence with the set of partitions $$\mathcal P(n+1) = \{\mathbf{d} =[d_1,d_2,...,d_l]\mid \sum_{i=1}^l d_i = n+1\}.$$ 

\item The nilpotent orbits in the Lie algebra of type $D_n$ are in correspondence with the set
$$\mathcal P_{+}(2n) = \{\mathbf{d}\in\mathcal P(2n)\mid \text{ even numbers occur even times in }\mathbf{d}\},$$
except that each of the \emph{very even partitions}, namely, a partition in which no odd number occurs, corresponds to two nilpotent orbits.

\item Furthermore, the partition corresponding to a nilpotent orbit is given by the sizes of Jordan blocks of the Jordan normal form of any representative.

\end{itemize}
      
\begin{notation}
  We use $[d_1^{r_1},d_2^{r_2},...,d_k^{r_k}]$ to denote the partition $\mathbf d$ where $d_1$ occurs $r_1$ times, $d_2$ occurs $r_2$ times, etc. When $r_i = 1$, we omit it. 
\end{notation}
      
\begin{example} Here are several examples to illustrate the classification, using partitions.

\begin{enumerate}
    \item For type $A_n$, we have $\mathfrak g = \mathfrak{sl}_{n+1}$ and the nilpotent elements are those whose eigenvalues are all $0$. By the theory of Jordan normal form, the nilpotent orbits are determined entirely by the sizes of its Jordan blocks. This gives a partition of $n+1$.
 
 \item For type $D_4$, we have $\mathfrak g =  \mathfrak{so}_8$. The set is $$\mathcal P_{+}(8) =\{[1^8],[1^6,2^2],[2^4],[1^5,3],[1^2,3^2],[4^2],[1^3,5],[1,7]\}.$$ The partitions $[2^4]$ and $[4^2]$ are very even partitions and each of them corresponds to $2$ nilpotent orbits. Together, there are $10$ nilpotent orbits inside $\mathfrak{so}_8$. As an example, $[4^2]$ is a very even partition. The nilpotent elements $X_{\alpha_1}+X_{\alpha_2}+X_{\alpha_3}$ and $X_{\alpha_1}+X_{\alpha_2}+X_{\alpha_4}$ are not conjugate under the adjoint $G$-action, but the respective matrices have the same Jordan normal form of partition $[4^2]$.
\end{enumerate}

\end{example}

\subsection{Calculate the partitions}
In this section, we introduce the algorithm to calculate the partition of a given nilpotent element in types $A$ and $D$.


\subsubsection{Type $A$} 
Let $E_{i,j}$ be the matrix with $1$ at $(i,j)$-entry and $0$ elsewhere. We choose the root vectors to be $X_{\alpha_i}=E_{i,i+1}$, $1\leq i\leq n$, and $X_{-\alpha_i}$ to be the transpose of $X_{\alpha_i}$.

\begin{algorithm} \label{algorithm:partitionA} 
    Given a labeled Dynkin diagram $\Gamma$ of type $A$, we construct a digraph $D$ and use it to determine the partition as follows.
    \begin{enumerate}[label=\textbf{Step \arabic*}.]
        \item Construct a graph $D'$ from $\Gamma$ as in the picture. Each vertex in $\Gamma$ corresponds to an edge in $D'$. We use the number of the vertex in $\Gamma$ to label the corresponding edge in $D'$. We also label the vertices in $D'$ as $v_1,\dots,v_{n+1}$. 

        \begin{center}
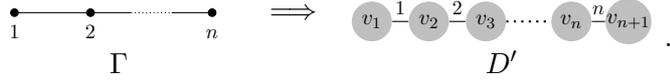
                \begin{tabular}{ c c c }
                    $$\dynkin[%
                    edge length=1cm,
                     labels={1,2,n}]A{**.*}
                    $$ & $\implies$ &
    
                    \begin{tikzpicture}[auto,scale=0.75,transform shape,-,main node/.style={circle,draw}, baseline=(current bounding box.center)]
                        \vertex (1) at (1,0) {$v_1$};
                        \vertex (2) at (2,0){$v_2$};
                        \vertex (3) at (3,0){$v_3$};
                        \vertex (4) at (4.5,0){$v_{n}$};
                        \vertex (5) at (5.5,0){$v_{n+1}$};
                        
                	\path[every node/.style={font=\sffamily\small}]
                	(1) edge node {$1$} (2)
                	(2) edge node {$2$} (3)
                	(3) edge[thick,dotted] (4)
                	(4) edge node {$n$} (5);
    
                        \end{tikzpicture}\\
                    $\Gamma$ & & $D'$\\
                \end{tabular}.
            \end{center}

        \item We assign arrows to edges in $D'$ by the following rules. If $v\in J_1\ (resp.\ J_2)$ is a vertex in $\Gamma$, draw a forward (resp. backward) arrow on the two corresponding edges in $D'$. Here, a forward arrow means drawing from the vertex of a smaller number to the other with a larger number, and a backward arrow means the opposite. We call this directed graph $D$.
        \item Repeat the following until the digraph $D$ is empty:
            \begin{enumerate}
                \item Let $C$ be the collection of longest directed paths of length $n_i$.
                \item Let $P$ be a maximal subset of $C$ consists of disjoint paths. Any choice of maximal subsets is feasible. Let $k_i$ be the cardinality of $P$.
                \item Delete $P$ from the digraph $D$. If a vertex is deleted, also delete all edges attached.
                \item Record $n_i+1$ for $k_i$ number of times in the resulting partition.
            \end{enumerate}
        \item Output the resulting partition.
    \end{enumerate}
\end{algorithm}

\begin{example} For type $A_{11}$, consider the following labeled Dynkin diagram
    $$\dynkin[%
    edge length=.75cm,
    labels={J_1,J_2,J_2,J_1,J_1,J_2,J_2,J_1,J_2,J_2,J_2}]A{11}.$$
    Then the above algorithm produces (we add some inclination to the edges so that the paths become apparent.)
    \begin{center}
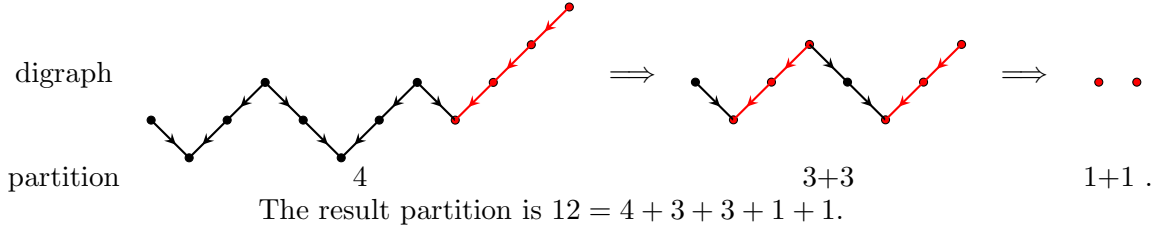
        \begin{tabular}{ c c c c c c}
         digraph  & $
            \tikz[scale=0.5,baseline=(current bounding box.center),decoration={markings, mark= at position 0.65 with {\arrow{stealth}}}]{
            \draw[fill=black] (-1,1) circle (3pt);
            \draw[fill=black] (0,0) circle (3pt);
            \draw[fill=black] (1,1) circle (3pt);
            \draw[fill=black] (2,2) circle (3pt);
            \draw[fill=black] (3,1) circle (3pt);
            \draw[fill=black] (4,0) circle (3pt);
            \draw[fill=black] (5,1) circle (3pt);
            \draw[fill=black] (6,2) circle (3pt);
            \draw[fill=red] (7,1) circle (3pt);
            \draw[fill=red] (8,2) circle (3pt);
            \draw[fill=red] (9,3) circle (3pt);
            \draw[fill=red] (10,4) circle (3pt);
            \draw[thick,postaction={decorate}] (-1,1) -- (0,0);
            \draw[thick,postaction={decorate}] (1,1) -- (0,0);
            \draw[thick,postaction={decorate}] (2,2) -- (1,1);
            \draw[thick,postaction={decorate}] (2,2) -- (3,1);
            \draw[thick,postaction={decorate}] (3,1) -- (4,0);
            \draw[thick,postaction={decorate}] (5,1) -- (4,0);
            \draw[thick,postaction={decorate}] (6,2) -- (5,1);
            \draw[thick,postaction={decorate}] (6,2) -- (7,1);
            \draw[thick,postaction={decorate},red] (8,2) -- (7,1);
            \draw[thick,postaction={decorate},red] (9,3) -- (8,2);
            \draw[thick,postaction={decorate},red] (10,4) -- (9,3);}$ & $\implies$ & 
            $\tikz[scale=0.5,baseline=(current bounding box.center),decoration={markings, mark= at position 0.65 with {\arrow{stealth}}}]{
            \draw[fill=black] (-1,1) circle (3pt);
            \draw[fill=red] (0,0) circle (3pt);
            \draw[fill=red] (1,1) circle (3pt);
            \draw[fill=red] (2,2) circle (3pt);
            \draw[fill=black] (3,1) circle (3pt);
            \draw[fill=red] (4,0) circle (3pt);
            \draw[fill=red] (5,1) circle (3pt);
            \draw[fill=red] (6,2) circle (3pt);
            \draw[thick,postaction={decorate}] (-1,1) -- (0,0);
            \draw[thick,postaction={decorate},red] (1,1) -- (0,0);
            \draw[thick,postaction={decorate},red] (2,2) -- (1,1);
            \draw[thick,postaction={decorate}] (2,2) -- (3,1);
            \draw[thick,postaction={decorate}] (3,1) -- (4,0);
            \draw[thick,postaction={decorate},red] (5,1) -- (4,0);
            \draw[thick,postaction={decorate},red] (6,2) -- (5,1);}$ &$\implies$&
            $\tikz[scale=0.5,baseline=(current bounding box.center),decoration={markings, mark= at position 0.65 with {\arrow{stealth}}}]{
            \draw[fill=red] (0,0) circle (3pt);
            \draw[fill=red] (1,0) circle (3pt);}$\\ 
           partition&4 & & 3+3 && 1+1 .\\
        \end{tabular}

 The result partition is $12=4+3+3+1+1$.
\end{center}
\end{example}

\begin{proposition}
    Given a labeled Dynkin diagram $(\Gamma,J_1,J_2)$ of type $A$, the partition of $X = \sum_{\alpha\in J_1}X_\alpha+\sum_{\beta\in J_2}X_{-\beta}$ is given by the above \cref{algorithm:partitionA}.
\end{proposition}

\begin{proof}
    Let $V=\mathbb{C}^n$ be standard representation of $\mathfrak{sl}_{n}$ and $\{e_j\}_{j=1}^n$ be the standard basis of $V$. Given a labeled Dynkin diagram, we will show that the partition of $X$ can be calculated as above. We will use $A$ to denote the matrix of $X$. 
    
     Consider the chain $$\{0\}\subsetneq \ker A\subsetneq \ker A^2\subsetneq \dots \subsetneq \ker A^r=V.$$ Note that if some collection $\{\bar{v_j}\}$ is linearly independent in $\ker A^j/\ker A^{j-1}$, then $\{\overline{Av_j}\}$ is linearly independent in $\ker A^{j-1}/\ker A^{j-2}$, where $\bar{v}$ denotes the image of $v\in V$ in respective quotient. To find the partition of $A$ is equivalent to finding the sequence of numbers $\dim \ker A^j/\ker A^{j-1}$, for $j=1,\dots r$, since $\dim \ker A^j/\ker A^{j-1}-\dim \ker A^{j+1}/\ker A^j$ is the number of Jordan blocks of size $j$.
    
    Define the \emph{height} of $v\in V$ to be the maximal $j$ such that its image $\bar{v}\in \ker A^j/\ker A^{j-1}$ is non-zero.  By definition, if $v$ is of height $j$, then $Av$ is of height $j-1$.
    
    Let $D$ be the digraph defined by the algorithm above. For simplicity, we write $1,\dots, n$ for vertices $v_1,\dots,v_n$ in $D$. It is easy to see $A$ is the adjacency matrix of $D$, i.e., $A_{lj}=1$ if there is an edge $l\to j$ in $D$ and $0$ otherwise. Hence, $(A^k)_{lj}$ corresponds to a path $l\to j$ in $D$ of length $k$. Moreover, every $e_j$ corresponds to the $j^{\text{th}}$ vertex in $D$.
    
    We induct on the length of the longest paths of the digraph $D$.
    
    Let $\{P_i: y_{i,1}\to\dots\to y_{i,r}\}_{i=1}^s$ be a maximal collection of \emph{disjoint} longest paths of length $r-1$ in $D$. Disjoint here means both vertices and edges are distinct. They are indeed non-empty and longest as $\ker A^{r-1}\subsetneq\ker A^r=V$.
    
    We claim that $\{\overline{e_{y_{i,r}}}\}_{i=1}^s$ forms a height $r$ basis in $\ker A^r/\ker A^{r-1}$. Indeed, by construction we have $e_{y_{i,r}}\in \ker A^r\setminus \ker A^{r-1}$ and $A^{r-1}e_{y_{i,r}}=e_{y_{i,1}}+\epsilon e_j$ for some $j\neq y_{i,1}$, where $\epsilon=0$ or $1$. Since $\{P_i\}_{i=1}^s$ are disjoint, the elements in $\{y_{i,r},y_{i,1}\ |\ i=1,\dots,s\}$ are distinct. Hence, $\Span \{e_{y_{i,r}}\} \cap \ker A^{r-1}=0$ and so $\{\overline{e_{y_{i,r}}}\}_{i=1}^s$ are linearly independent. Furthermore, every non-zero entry in $A^{r-1}$ corresponds to a path of length $r-1$ so linearly independent columns of $A^{r-1}$ correspond to exactly end vertices of some paths, i.e., $y_{i,r}$ of some $P_i$. As a result, $\{\overline{e_{y_{i,r}}}\}_{i=1}^s$ spans $\ker A^r/\ker A^{r-1}$. Therefore, we established a one-to-one correspondence between a maximal collection of disjoint longest paths in $\Gamma$ and a basis of $\ker A^r/\ker A^{r-1}$, mapping $P_i$ to $\overline{e_{y_{i,r}}}$.

    Next, we consider the induction step. As $\{\overline{e_{y_{i,r}}}\}_{i=1}^s$ is a basis of $\ker A^r/\ker A^{r-1}$, $\{\overline{Ae_{y_{i,r}}}\}_{i=1}^s$ is a linearly independent set in $\ker A^{r-1}/\ker A^{r-2}$. We want to extend it to a basis and show the extension elements correspond to a maximal collection of disjoint paths in $D\setminus \{P_i\}_{i=1}^s$. (We also need to delete all edges attached to the deleted vertices.)
    
    Let $\{Q_i: y_{i,1}\to\dots\to y_{i,r-1}\}_{i=1}^s$ be the collection of paths $Q_i$ each of which is $P_i$ without the terminal vertex. Then $Q_i$'s are disjoint. Consider a maximal collection of paths $\{Q_t: y_{t,1},\dots, y_{t,r-1}\}_{t=s+1}^{s+m}$ disjoint from $\{P_i\}_{i=1}^s$. Define $e_{y_{j,r-k}}' \defeq A^k e_{y_{j,r}}$ for $k=1,\dots, r-1$, $j=1,\dots,s+m$ and $e_l'=e_l$ for the rest of $l$ not equals to any $y_{j,r-k}$. Observe that $\{e_l'\}_{l=1}^n$ is a basis for $V$. This is easy to see from the graph if we relabel the vertex $l$ in $D$ to $j_1+j_2\dots$ for $e_l'=e_{j_1}+e_{j_2}+\dots$. Call this new digraph $D^\prime$. 
    
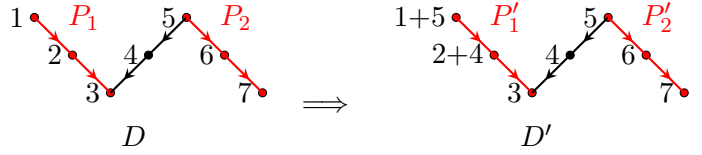
    \begin{center}
    \begin{adjustbox}{max width=\textwidth}
        \begin{tabular}{c c c}
            $\tikz[scale=0.5,decoration={markings, mark= at position 0.65 with {\arrow{stealth}}}]{
                \draw[fill=red] (0,2) circle (3pt)node[anchor=east]{1};
                \draw[fill=red] (1,1) circle (3pt) node[anchor=east]{2};
                \draw[fill=red] (2,0) circle (3pt)node[anchor=east]{3};
                \draw[fill=black] (3,1) circle (3pt) node[anchor=east]{4};
                \draw[fill=red] (4,2) circle (3pt) node[anchor=east]{5};
                \draw[fill=red] (5,1) circle (3pt) node[anchor=east]{6};
                \draw[fill=red] (6,0) circle (3pt) node[anchor=east]{7};
                \draw[thick,postaction={decorate},red] (0,2) node[anchor=west, xshift=0.3cm]{$P_1$} -- (1,1) ;
                \draw[thick,postaction={decorate},red] (1,1) -- (2,0);
                \draw[thick,postaction={decorate}] (3,1) -- (2,0);
                \draw[thick,postaction={decorate}] (4,2) -- (3,1);
                \draw[thick,postaction={decorate},red] (4,2) node[anchor=west,xshift=0.3cm]{$P_2$} -- (5,1) ;
                \draw[thick,postaction={decorate},red] (5,1) -- (6,0);
            }$
            &$\implies$&                
            $\tikz[scale=0.5,decoration={markings, mark= at position 0.65 with {\arrow{stealth}}}]{
                \draw[fill=red] (0,2) circle (3pt)node[anchor=east]{1+5};
                \draw[fill=red] (1,1) circle (3pt) node[anchor=east]{2+4};
                \draw[fill=red] (2,0) circle (3pt)node[anchor=east]{3};
                \draw[fill=black] (3,1) circle (3pt) node[anchor=east]{4};
                \draw[fill=red] (4,2) circle (3pt) node[anchor=east]{5};
                \draw[fill=red] (5,1) circle (3pt) node[anchor=east]{6};
                \draw[fill=red] (6,0) circle (3pt) node[anchor=east]{7};
                \draw[thick,postaction={decorate},red] (0,2) node[anchor=west, xshift=0.3cm]{$P_1^\prime$} -- (1,1) ;
                \draw[thick,postaction={decorate},red] (1,1) -- (2,0);
                \draw[thick,postaction={decorate}] (3,1) -- (2,0);
                \draw[thick,postaction={decorate}] (4,2) -- (3,1);
                \draw[thick,postaction={decorate},red] (4,2) node[anchor=west,xshift=0.3cm]{$P_2^\prime$} -- (5,1) ;
                \draw[thick,postaction={decorate},red] (5,1) -- (6,0);
            }$
                \\
                $D$ & & $D^\prime$\\
        \end{tabular}.
    \end{adjustbox}
    \end{center}
    
    Now 
    \begin{equation*}
    \begin{cases}
      e_j'\notin \ker A^{r-1} & \text{if}\ j=y_{i,r} \text{ for some}\  i=1,\dots s; \\
      e_j'\in \ker A^{r-1}\setminus \ker A^{r-2} & \text{if}\ j=y_{i,r-1} \text{ for some}\  i=1,\dots s+m; \\
      e_j'\in \ker A^{r-2} & \text{otherwise.}
    \end{cases}
  \end{equation*}
    
    Hence, we have $\{\overline{e_{y_{i,r-1}}'}\}_{i=1}^{s+m}$ is a basis in $\ker A^{r-1}/\ker A^{r-2}$. As $\{Q_t\}_{t=s+1}^{s+m}$ are disjoint from $\{P_i\}_{i=1}^s$, it is equivalent to consider the maximal collection of disjoint paths in $D\setminus \{P_i\}_{i=1}^s$. We have shown that extending $\{\overline{Ae_{y_{i,r}}}\}_{i=1}^s$ to a basis in $\ker A^{r-1}/\ker A^{r-2}$ is equivalent to considering a maximal collection of disjoint paths in $D\setminus \{P_i\}_{i=1}^s$. In particular, the number of extension vectors and the new paths are both $m$. Hence, we repeat the proof on smaller graphs and it completes the proof.
\end{proof}

\subsubsection{Type $D$}
We choose the root vectors to be $X_{\alpha_i}=E_{i,i+1}-E_{i+1+n,i+n}$, $1\leq i\leq n-1$; $X_{\alpha_n}=E_{n-1,2n}-E_{n,2n-1}$, and $X_{-\alpha_i}$ to be the transpose of $X_{\alpha_i}$.
    \begin{algorithm}\label{algorithm:partitionD} 
        Given a labeled Dynkin diagram $\Gamma$ of type $D$, we construct a digraph $D$ and use it to determine the partition as follows.
        
        \begin{enumerate}[label=\textbf{Step \arabic*}.]
            \item Construct a graph $D'$ from $\Gamma$ as in the picture. Each vertex in $\Gamma$ corresponds to exactly two edges in $D'$. We use the number of the vertex in $\Gamma$ to label the corresponding two edges in $D'$. We also label the vertices in $D'$ as $v_1,\dots,v_{2n}$.

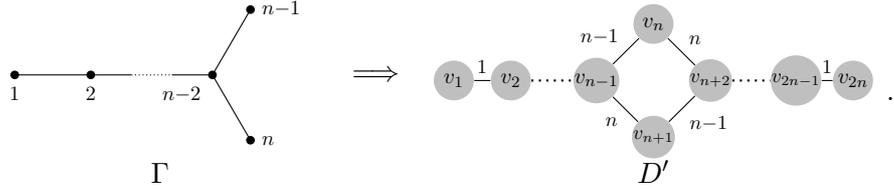
            \begin{center}
                \begin{tabular}{ c c c }
                    $$\dynkin[%
                    edge length=1cm,
                     labels={1,2,n{\tiny -}2,n{\tiny -}1,n}]D{**.***}
                    $$ & $\implies$ &
    
                    \begin{tikzpicture}[auto,scale=0.75,transform shape,-,main node/.style={circle,draw}, baseline=(current bounding box.center)]
                        \vertex (1) at (1,0) {$v_1$};
                        \vertex (2) at (2,0){$v_2$};
                        \vertex (3) at (3.5,0){ $v_{n-1}$};
                        \vertex (4) at (4.5,1){$v_n$};
                        \vertex (5) at (4.5,-1){\small $v_{n+1}$};
                        \vertex (6) at (5.5,0){\small $v_{n+2}$};
                        \vertex (7) at (7,0) {\small $v_{2n-1}$};
                        \vertex (8) at (8,0){$v_{2n}$};
                	\path[every node/.style={font=\sffamily\small}]
                	(1) edge node {$1$} (2)
                	(2) edge[thick,dotted] (3)
                	(3) edge node {$n{\tiny -}1$} (4)
                	(3) edge node [below left] {$n$} (5)
                	(4) edge node {$n$} (6)
                	(5) edge node [below right] {$n{\tiny -}1$} (6)
                	(6) edge[thick,dotted] (7)
                	(7) edge node {$1$} (8);
    
                        \end{tikzpicture}\\
                    $\Gamma$ & & $D'$\\
                \end{tabular}.
            \end{center}

             \item We assign arrows to edges in $D'$ by the following rules. If $v\in J_1\ (resp.\ J_2)$ is a vertex in $\Gamma$, draw a forward (resp. backward) arrow on the two corresponding edges in $D'$. Here, a forward arrow means drawing from the vertex of a smaller number to the other with a larger number, and a backward arrow means the opposite. We call this directed graph $D$.
             
            \item Repeat the following until the digraph $D$ is empty:
                \begin{enumerate}
                    \item Let $C$ be the collection of longest directed paths of length $n_i$.
                    \item Let $P$ be a maximal subset of $C$ consists of disjoint paths. Any choice of maximal subsets is feasible. Let $k_i$ be the cardinality of $P$.
                    \item Delete $P$ from the digraph $D$. If a vertex is deleted, also delete all edges attached.
                    \item Record $n_i+1$ for $k_i$ number of times in the resulting partition.
                \end{enumerate}
            \item Output the resulting partition.
        \end{enumerate}
    \end{algorithm}
\begin{example} For type $D_{6}$, consider the following labeled Dynkin diagram
    $$\dynkin[%
    edge length=.75cm,
    labels={J_2,J_2,J_1,J_1,J_1,J_2}]D{6}.$$
    Then the above algorithm produces (we add some inclination to the edges so that the paths become apparent.)
    \begin{center}
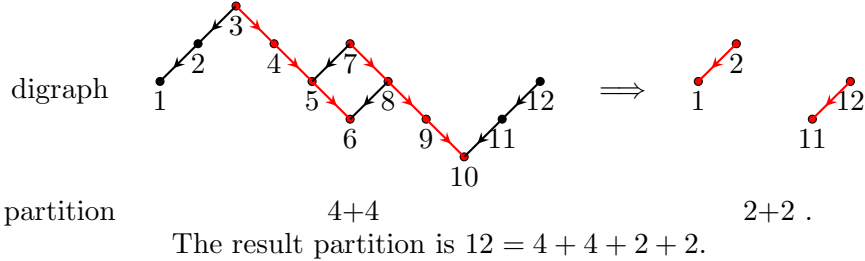
\begin{tabular}{ c c c c}
   digraph &$
    \tikz[scale=0.5,baseline=(current bounding box.center),decoration={markings, mark= at position 0.65 with {\arrow{stealth}}}]{
    \draw[fill=black] (-5,1) circle (3pt) node[anchor=north]{1};
    \draw[fill=black] (-4,2) circle (3pt) node[anchor=north]{2};
    \draw[fill=red] (-3,3) circle (3pt) node[anchor=north]{3};
    \draw[fill=red] (-2,2) circle (3pt) node[anchor=north]{4};
    \draw[fill=red] (-1,1) circle (3pt) node[anchor=north]{5};
    \draw[fill=red] (0,0) circle (3pt) node[anchor=north]{6};
    \draw[fill=red] (0,2) circle (3pt) node[anchor=north]{7};
    \draw[fill=red] (1,1) circle (3pt) node[anchor=north]{8};
    \draw[fill=red] (2,0) circle (3pt) node[anchor=north]{9};
    \draw[fill=red] (3,-1) circle (3pt) node[anchor=north]{10};
    \draw[fill=black] (4,0) circle (3pt) node[anchor=north]{11};
    \draw[fill=black] (5,1) circle (3pt) node[anchor=north]{12};
    \draw[thick,postaction={decorate}] (-4,2) -- (-5,1);
    \draw[thick,postaction={decorate}] (-3,3) -- (-4,2);
    \draw[thick,postaction={decorate},red] (-3,3) -- (-2,2);
    \draw[thick,postaction={decorate},red] (-2,2) -- (-1,1);
    \draw[thick,postaction={decorate},red] (-1,1) -- (0,0);
    \draw[thick,postaction={decorate}]  (0,2) -- (-1,1) ;
    \draw[thick,postaction={decorate},red] (0,2) -- (1,1);
    \draw[thick,postaction={decorate}]  (1,1)-- (0,0);
    \draw[thick,postaction={decorate},red] (1,1) -- (2,0);
    \draw[thick,postaction={decorate},red] (2,0) -- (3,-1);
    \draw[thick,postaction={decorate}] (4,0) -- (3,-1);
    \draw[thick,postaction={decorate}] (5,1) -- (4,0);

    }$ & $\implies$ & $
    \tikz[scale=0.5,baseline=(current bounding box.center),decoration={markings, mark= at position 0.65 with {\arrow{stealth}}}]{
    \draw[fill=red] (-2,1) circle (3pt) node[anchor=north]{1};
    \draw[fill=red] (-1,2) circle (3pt) node[anchor=north]{2};
    \draw[fill=red] (1,0) circle (3pt) node[anchor=north]{11};
    \draw[fill=red] (2,1) circle (3pt) node[anchor=north]{12};
    \draw[thick,postaction={decorate},red] (-1,2) -- (-2,1);
    \draw[thick,postaction={decorate},red] (2,1) -- (1,0);

    }$\\ 
 partition&4+4 && 2+2 .\\
\end{tabular}

 The result partition is $12=4+4+2+2$.
\end{center}
\end{example}

\begin{proposition}
     Given a labeled Dynkin diagram $(\Gamma,J_1,J_2)$ of type $D$, the partition of $X = \sum_{\alpha\in J_1}X_\alpha+\sum_{\beta\in J_2}X_{-\beta}$ is given by the above \cref{algorithm:partitionD}.
   
\begin{proof}
We will give a proof similar to the type $A$ case. Again for simplicity, we use numbers $1,\dots, 2n$ for vertices $v_1,\dots,v_{2n}$ in $D$.

For such $X$, for $n+1\leq j\leq 2n$, exchange the $j^\text{th}$ column with the $(3n-j+1)^\text{th}$ column and the $j^\text{th}$ row with the $(3n-j+1)^\text{th}$ row. In order words, define a new matrix $A\defeq\gamma(X)$, where $\gamma$ is defined to be a linear map satisfying 
$$\gamma(X_{\alpha_j})=E_{j,j+1}-E_{2n-j,2n-j+1}, 1\leq j\leq n-1;\ \ \gamma(X_{\alpha_n})= E_{n-1,n+1}-E_{n,n+2}. $$ Note that this map preserves Jordan normal forms. Let $D$ be the digraph constructed as in the algorithm. Then $A$ is the ``signed" adjacency matrix of $D$, namely,

\begin{equation*}
  A_{ij}=
    \begin{cases}
      1, & \text{if $l\to j$ is an edge in $D$, where $l,j\in$\{$1,\dots, n, n+1$\} are node labels in $D$,}\\
      -1, & \text{if $l\to j$ is an edge in $D$, where $l,j\in$\{$n,n+1,\dots, 2n$\} are node labels in $D$,}\\
      0, & \text{otherwise.}
    \end{cases}       
\end{equation*}

Now we induct on the length of the longest path, similar to what we have done in type $A$. Let $\{P_i:y_{i,1}\to\cdots\to y_{i,r}\}_{i=1}^s$ be a maximal collection of disjoint longest paths in $D$, where $r$ is the smallest number such that $A^r = 0$. Again, \textit{disjoint} means both vertices and edges are distinct. We claim that $\{\overline{e_{y_{i,r}}}\}_{i=1}^s$ is a basis of $\ker A^r/\ker A^{r-1}$. Indeed, by similar reasoning to type $A$ they are linearly independent and $e_{y_{i,r}}\in \ker A^r\setminus \ker A^{r-1}$. It remains to show $\{A^{r-1}e_{y_{i,r}}\}_{i=1}^s$ are linearly independent so that $s=\dim \Im A^{r-1}=\dim \ker A^r/\ker A^{r-1}$. 

Observe that the $l$-th coordinate of $|A^{r-1}e_j|$ is $(|A^{r-1}|)_{lj}$ which is the number of paths of length $r-1$ from $l$ to $j$. Here $|\cdot|$ is the absolute value. These paths are denoted by $Q_{j,l,1}, Q_{j,l,2},\cdots$. We claim that $\{y_{i,r}\}_{i=1}^s$ can be chosen from $\{j\mid A^{r-1}e_j\neq 0\}$ such that $\{A^{r-1}e_{y_{i,r}}\}$ is a basis for $\Im A^{r-1}$, equivalently, disjoint $P_i$ can be chosen from $\{Q_{j,l,k}\}_k$ above. It is only critical to check the cases where some $Q_{j,l,k}$ is involved in the middle ``diamond" of $D$, otherwise, the proof is identical to that of type $A$.

Note that the Jordan normal form will not change if we switch $J_1$ and $J_2$ (transpose $A$). Also by symmetry, switching the last two vertices in $\Gamma$ will not add more cases. We only have three cases for the last three vertices in $\Gamma$.

{
\centering
\begin{longtable}[H]{| c | c | c | M{6cm} |}
    \hline
    Case & $\Gamma$ & $D$ &  \\
    \hline
    
    1 & \tikz[scale=0.5,decoration={markings, mark= at position 0.5 with {\arrow{stealth}}},baseline=(current bounding box.center),every circle node/.style={draw}]{
        \draw[fill=black] (1,0) circle (3pt) node[anchor=south] {$J_1$};
        \draw[fill=black] (2,1) circle (3pt) node[anchor=south] {$J_1$};
        \draw[fill=black] (2,-1) circle (3pt) node[anchor=south] {$J_1$};
        \draw[thick,dotted]  (0,0) -- (1,0);
        \draw (1,0) -- (2,1);
        \draw (1,0) -- (2,-1);
    }
             & 
             
        \tikz[scale=0.5,decoration={markings, mark= at position 0.5 with {\arrow{stealth}}},baseline=-25pt,every circle node/.style={draw}]{
        \draw[fill=black] (0,0) circle (3pt) node[anchor=south] {$l$};
        \draw[fill=black] (1,-1) circle (3pt);
        \draw[fill=black] (2,-2) circle (3pt);
        \draw[fill=black] (3.5,-2.5) circle (3pt);
        \draw[fill=black] (2.5,-3.5) circle (3pt);
        \draw[fill=black] (4,-4) circle (3pt);
        \draw[fill=black] (5,-5) circle (3pt);
        \draw[fill=black] (6,-6) circle (3pt) node[anchor=south] {$j$};
        \draw[thick,dotted]  (0,0) -- (1,-1);
        \draw[postaction={decorate}] (1,-1) -- (2,-2);
        \draw[postaction={decorate}] (2,-2) -- (3.5,-2.5);
        \draw[postaction={decorate}] (2,-2) -- (2.5,-3.5);
        \draw[postaction={decorate}] (3.5,-2.5) -- (4,-4);
        \draw[postaction={decorate}] (2.5,-3.5) -- (4,-4);
        \draw[postaction={decorate}] (4,-4) -- (5,-5);
        \draw[thick,dotted,] (5,-5) -- (6,-6);
    } & 
        In this case $\{Q_{j,l,k}\}_k$ must all go through the ``diamond". We may choose $j=y_{i,r}$ for some $i$.
    \\ \hline

    2.1 & \tikz[scale=0.5,decoration={markings, mark= at position 0.5 with {\arrow{stealth}}},baseline=(current bounding box.center),every circle node/.style={draw}]{
        \draw[fill=black] (1,0) circle (3pt) node[anchor=south] {$J_1$};
        \draw[fill=black] (2,1) circle (3pt) node[anchor=south] {$J_1$};
        \draw[fill=black] (2,-1) circle (3pt) node[anchor=south] {$J_2$};
        \draw[thick,dotted]  (0,0) -- (1,0);
        \draw (1,0) -- (2,1);
        \draw (1,0) -- (2,-1);
    }
             & 
             
        \tikz[scale=0.5,decoration={markings, mark= at position 0.5 with {\arrow{stealth}}},baseline=-25pt,every circle node/.style={draw}]{
        \draw[fill=black] (0,0) circle (3pt) node[anchor=south] {$l_1$};
        \draw[fill=black] (1,-1) circle (3pt);
        \draw[fill=black] (2,-2) circle (3pt);
        \draw[fill=black] (3,-1) circle (3pt) node[anchor=south] {$l_2$};
        \draw[fill=black] (3,-3) circle (3pt) node[anchor=north] {$j_1$};
        \draw[fill=black] (4,-2) circle (3pt);
        \draw[fill=black] (5,-3) circle (3pt);
        \draw[fill=black] (6,-4) circle (3pt) node[anchor=south] {$j_2$};
        \draw[thick,dotted]  (0,0) -- (1,-1);
        \draw[postaction={decorate}] (1,-1) -- (2,-2);
        \draw[postaction={decorate}] (3,-1) -- (2,-2);
        \draw[postaction={decorate}] (2,-2) -- (3,-3);
        \draw[postaction={decorate}] (3,-1) -- (4,-2);
        \draw[postaction={decorate}] (4,-2) -- (3,-3);
        \draw[postaction={decorate}] (4,-2) -- (5,-3);
        \draw[thick,dotted] (5,-3) -- (6,-4);
    } & 
        {
        \centering
        \begin{tabular}{c|c c}
            $(|A^{r-1}|)_{lj}$ & $j_1$ & $j_2$ \\
            \hline
             $l_1$ & $1$ & $0$ \\
             $l_2$ & $2$ & $1$ \\
        \end{tabular}
        }
        
        Therefore $\{A^{r-1}e_{j_1},A^{r-1}e_{j_2}\}$ are linearly independent and $j_1, j_2$ are in some $\{y_{i,r}\}_i$.
    \\ \hline

    2.2 & \tikz[scale=0.5,decoration={markings, mark= at position 0.5 with {\arrow{stealth}}},baseline=(current bounding box.center),every circle node/.style={draw}]{
        \draw[fill=black] (1,0) circle (3pt) node[anchor=south] {$J_1$};
        \draw[fill=black] (2,1) circle (3pt) node[anchor=south] {$J_1$};
        \draw[fill=black] (2,-1) circle (3pt) node[anchor=south] {$J_2$};
        \draw[thick,dotted]  (0,0) -- (1,0);
        \draw (1,0) -- (2,1);
        \draw (1,0) -- (2,-1);
    }
             & 
             
        \tikz[scale=0.5,decoration={markings, mark= at position 0.5 with {\arrow{stealth}}},baseline=-25pt,every circle node/.style={draw}]{
        \draw[fill=black] (-1,-1) circle (3pt);
        \draw[fill=black] (0,-2) circle (3pt);
        \draw[fill=black] (1,-1) circle (3pt) node[anchor=south] {$l_1$};
        \draw[fill=black] (2,-2) circle (3pt);
        \draw[fill=black] (3,-1) circle (3pt) node[anchor=south] {$l_2$};
        \draw[fill=black] (3,-3) circle (3pt) node[anchor=north] {$j_1$};
        \draw[fill=black] (4,-2) circle (3pt);
        \draw[fill=black] (5,-3) circle (3pt) node[anchor=north] {$j_2$};
        \draw[fill=black] (6,-2) circle (3pt);
        \draw[fill=black] (7,-3) circle (3pt);
        \draw[thick,dotted]   (-1,-1) -- (0,-2) ;
        \draw[postaction={decorate}] (1,-1) -- (0,-2);
        \draw[postaction={decorate}] (1,-1) -- (2,-2);
        \draw[postaction={decorate}] (3,-1) -- (2,-2);
        \draw[postaction={decorate}] (2,-2) -- (3,-3);
        \draw[postaction={decorate}] (3,-1) -- (4,-2);
        \draw[postaction={decorate}] (4,-2) -- (3,-3);
        \draw[postaction={decorate}] (4,-2) -- (5,-3);
        \draw[postaction={decorate}]  (6,-2) -- (5,-3) ;
        \draw[thick,dotted]  (6,-2) -- (7,-3) ;
    } & 
        {
        \centering
        \begin{tabular}{c|c c}
            $(|A^{r-1}|)_{lj}$ & $j_1$ & $j_2$ \\
            \hline
             $l_1$ & $1$ & $0$ \\
             $l_2$ & $2$ & $1$ \\
        \end{tabular}
        }
        
        Therefore $\{A^{r-1}e_{j_1},A^{r-1}e_{j_2}\}$ are linearly independent and $j_1, j_2$ are in some $\{y_{i,r}\}_i$ independent of the choice of the rest $\{y_{i,r}\}_i$ in $D$.
    \\ \hline

    2.3 & \tikz[scale=0.5,decoration={markings, mark= at position 0.5 with {\arrow{stealth}}},baseline=(current bounding box.center),every circle node/.style={draw}]{
        \draw[fill=black] (1,0) circle (3pt) node[anchor=south] {$J_1$};
        \draw[fill=black] (2,1) circle (3pt) node[anchor=south] {$J_1$};
        \draw[fill=black] (2,-1) circle (3pt) node[anchor=south] {$J_2$};
        \draw[thick,dotted]  (0,0) -- (1,0);
        \draw (1,0) -- (2,1);
        \draw (1,0) -- (2,-1);
    }
             & 
             
        \tikz[scale=0.5,decoration={markings, mark= at position 0.5 with {\arrow{stealth}}},baseline=-25pt,every circle node/.style={draw}]{
        \draw[fill=black] (-2,-2) circle (3pt) ;
        \draw[fill=black] (-1,-3) circle (3pt) node[anchor=north] {$j_1$};
        \draw[fill=black] (0,-2) circle (3pt);
        \draw[fill=black] (1,-1) circle (3pt) node[anchor=south] {$l_1$};
        \draw[fill=black] (2,-2) circle (3pt);
        \draw[fill=black] (3,-1) circle (3pt) node[anchor=south] {$l_2$};
        \draw[fill=black] (3,-3) circle (3pt) node[anchor=north] {$j_2$};
        \draw[fill=black] (4,-2) circle (3pt);
        \draw[fill=black] (5,-3) circle (3pt) node[anchor=north] {$j_3$};
        \draw[fill=black] (6,-2) circle (3pt);
        \draw[fill=black] (7,-1) circle (3pt)  node[anchor=south] {$l_3$};
        \draw[fill=black] (8,-2) circle (3pt);
        \draw[thick,dotted]   (-1,-3) -- (-2,-2) ;
        \draw[postaction={decorate}] (0,-2) -- (-1,-3);
        \draw[postaction={decorate}] (1,-1) -- (0,-2);
        \draw[postaction={decorate}] (1,-1) -- (2,-2);
        \draw[postaction={decorate}] (3,-1) -- (2,-2);
        \draw[postaction={decorate}] (2,-2) -- (3,-3);
        \draw[postaction={decorate}] (3,-1) -- (4,-2);
        \draw[postaction={decorate}] (4,-2) -- (3,-3);
        \draw[postaction={decorate}] (4,-2) -- (5,-3);
        \draw[postaction={decorate}]  (6,-2) -- (5,-3);
        \draw[postaction={decorate}]  (7,-1) -- (6,-2);
        \draw[thick,dotted]  (8,-2) -- (7,-1) ;
    } & 
        {
        \centering
        \begin{tabular}{c|c c c}
            $(|A^{r-1}|)_{lj}$ & $j_1$ & $j_2$ & $j_3$\\
            \hline
             $l_1$ & $1$ & $1$ & $0$ \\
             $l_2$ & $0$ & $2$ & $1$ \\
             $l_3$ & $0$ & $0$ & $1$ \\
        \end{tabular}
        }
        
        Therefore $\{A^{r-1}e_{j_1},A^{r-1}e_{j_2},A^{r-1}e_{j_3}\}$ are linearly independent and $j_1, j_2,j_3$ are in some $\{y_{i,r}\}_i$ independent of the choice of the rest $ \{y_{i,r}\}_i$ in $D$.
    \\ \hline

    3.1 & \tikz[scale=0.5,decoration={markings, mark= at position 0.5 with {\arrow{stealth}}},baseline=(current bounding box.center),every circle node/.style={draw}]{
        \draw[fill=black] (1,0) circle (3pt) node[anchor=south] {$J_1$};
        \draw[fill=black] (2,1) circle (3pt) node[anchor=south] {$J_2$};
        \draw[fill=black] (2,-1) circle (3pt) node[anchor=south] {$J_2$};
        \draw[thick,dotted]  (0,0) -- (1,0);
        \draw (1,0) -- (2,1);
        \draw (1,0) -- (2,-1);
    }
             & 
             
        \tikz[scale=0.5,decoration={markings, mark= at position 0.5 with {\arrow{stealth}}},baseline=-25pt,every circle node/.style={draw}]{
        \draw[fill=black] (0,0) circle (3pt) node[anchor=south] {$l_1$};
        \draw[fill=black] (1,-1) circle (3pt);
        \draw[fill=black] (2,-2) circle (3pt) node[anchor=north] {$j_1$};
        \draw[fill=black] (2.5,-0.5) circle (3pt);
        \draw[fill=black] (3.5,-1.5) circle (3pt);
        \draw[fill=black] (4,0) circle (3pt)  node[anchor=south] {$l_2$};
        \draw[fill=black] (5,-1) circle (3pt);
        \draw[fill=black] (6,-2) circle (3pt) node[anchor=south] {$j_2$};
        \draw[thick,dotted]  (0,0) -- (1,-1);
        \draw[postaction={decorate}] (1,-1) -- (2,-2);
        \draw[postaction={decorate}] (2.5,-0.5) -- (2,-2);
        \draw[postaction={decorate}] (3.5,-1.5) -- (2,-2);
        \draw[postaction={decorate}] (4,0) -- (2.5,-0.5);
        \draw[postaction={decorate}] (4,0) -- (3.5,-1.5);
        \draw[postaction={decorate}] (4,0) -- (5,-1);
        \draw[thick,dotted] (5,-1) -- (6,-2);
    } & 
        {
        \centering
        \begin{tabular}{c|c c}
            $(|A^{r-1}|)_{lj}$ & $j_1$ & $j_2$ \\
            \hline
             $l_1$ & $1$ & $0$ \\
             $l_2$ & $2$ & $1$ \\
        \end{tabular}
        }
        
        Therefore $\{A^{r-1}e_{j_1},A^{r-1}e_{j_2}\}$ are linearly independent and $j_1, j_2$ are in some $\{y_{i,r}\}_i$.
    \\ \hline

    3.2 & \tikz[scale=0.5,decoration={markings, mark= at position 0.5 with {\arrow{stealth}}},baseline=(current bounding box.center),every circle node/.style={draw}]{
        \draw[fill=black] (1,0) circle (3pt) node[anchor=south] {$J_1$};
        \draw[fill=black] (2,1) circle (3pt) node[anchor=south] {$J_2$};
        \draw[fill=black] (2,-1) circle (3pt) node[anchor=south] {$J_2$};
        \draw[thick,dotted]  (0,0) -- (1,0);
        \draw (1,0) -- (2,1);
        \draw (1,0) -- (2,-1);
    }
             & 
             
        \tikz[scale=0.5,decoration={markings, mark= at position 0.5 with {\arrow{stealth}}},baseline=-25pt,every circle node/.style={draw}]{
        \draw[fill=black] (0,-2) circle (3pt);
        \draw[fill=black] (1,-1) circle (3pt);
        \draw[fill=black] (2,-2) circle (3pt) node[anchor=north] {$j$};
        \draw[fill=black] (2.5,-0.5) circle (3pt);
        \draw[fill=black] (3.5,-1.5) circle (3pt);
        \draw[fill=black] (4,0) circle (3pt)  node[anchor=south] {$l$};
        \draw[fill=black] (5,-1) circle (3pt);
        \draw[fill=black] (6,0) circle (3pt);
        \draw[thick,dotted] (1,-1) -- (0,-2);
        \draw[postaction={decorate}] (1,-1) -- (2,-2);
        \draw[postaction={decorate}] (2.5,-0.5) -- (2,-2);
        \draw[postaction={decorate}] (3.5,-1.5) -- (2,-2);
        \draw[postaction={decorate}] (4,0) -- (2.5,-0.5);
        \draw[postaction={decorate}] (4,0) -- (3.5,-1.5);
        \draw[postaction={decorate}] (4,0) -- (5,-1);
        \draw[thick,dotted] (6,0) -- (5,-1);
    } & 
        In this case $\{Q_{j,l,k}\}_k$ must all go through the ``diamond". We may choose $j=y_{i,r}$ for some $i$.
    \\ \hline
\end{longtable}
}

  From the above discussion, we always have $\{j:A^{r-1}e_j\neq 0\}=\{y_{i,r}\}_i$ by merging the choices in the ``diamond" part and other parts. Hence, $\{\overline{e_{y_{i,r}}}\}_{i=1}^s$ is a basis of $\ker A^r/\ker A^{r-1}$.

  For the induction step, the same proof for type $A$ applies. 
\end{proof}
\end{proposition}

\begin{remark}\label{rmk: partitiontype}
    Observe that the partition produced by the above algorithm satisfies the following properties: either one odd number $\neq 1$ occurs in odd multiplicity while $1$ also occurs in odd multiplicity, or every number occurs in even multiplicity.
\end{remark}


    \subsection{From partition to the minimal Levi subalgebras}
    
    In this section, we will complete the proof of \cref{thm:regularity} for type $A$ and $D$. For every partition $\mathbf{d}$ that occurs in the last section, we construct a nilpotent element $\tilde{X}$ with partition $\mathbf d$ that is regular in some Levi subalgebra $\mathfrak l_J$. To see that $\tilde{X}$ here indeed has the expected partition, one can construct the respective $\mathfrak{sl}_2$ triples and look at the eigenvalues of the neutral elements. We refer to \cite[\S5.2]{CM} for details. 

\subsubsection{Type $A$} \label{constructA}
    The conclusion of the theorem is obvious in type $A$ since every nilpotent element is regular in its minimal Levi subalgebra. However, we still provide the construction here.
    
    Given $\mathbf{d}=[d_1,d_2,...,d_l]$, consider
    \begin{align*} 
        &J\defeq \{\alpha_1,\alpha_2,\cdots,\alpha_{d_1-1},\widehat{{\alpha}_{d_1}},\alpha_{d_1+1},\cdots,\alpha_{d_1+d_2-1},\widehat{{\alpha}_{d_1+d_2}},\alpha_{d_1+d_2+1},\cdots\}, \\
        &\tilde{X}\defeq \sum_{\alpha\in J}X_\alpha .
    \end{align*}
    Here the hat $\widehat\alpha$ means omission. Then $\tilde{X}$ has the partition $\mathbf d$ and is regular in $\mathfrak l_J$. The Levi subgroup $L_J$ (and its Lie algebra $\mathfrak l_J$) is of type $\prod_{i=1}^l A_{d_i-1}$.




\subsubsection{Type $D$} \label{constructD}
  The partition is of the form either $\mathbf d_1 = [(k_1+1)^2,(k_2+1)^2,...,(k_p+1)^2,2q-1,1]$ or $\mathbf d_2 = [(k_1+1)^2,(k_2+1)^2,...,(k_r+1)^2]$ where $r\geq 1$ (\cref{rmk: partitiontype}).

\begin{enumerate}

 \item For $\mathbf{d}=\mathbf{d_1}$, take
     \begin{align*}
     &J \defeq\{\alpha_1,\cdots,\alpha_{k_1-1},\widehat{{\alpha}_{k_1}},\alpha_{k_1+1},\cdots,\alpha_{k_1+k_2-1},\widehat{{\alpha}_{k_1+k_2}},\alpha_{k_1+k_2+1},\cdots,\widehat{\alpha_{n-q}},\alpha_{n-q+1},...,\alpha_{n-1},\alpha_n\},\\
     &\tilde{X}\defeq \sum_{\alpha\in J}X_{\alpha} 
     \end{align*}
     Then $\tilde{X}$ has the partition $\mathbf d_1$ and is regular in $\mathfrak l_J$. The Levi subgroup $L_J$ (and its Lie algebra) is of type $\prod_{i=1}^pA_{k_i}\times D_q$.
     
  \item For $\mathbf{d}=\mathbf{d_2}$, and $\mathbf d_2$ not a very even partition (\ref{partition}). Take
  \begin{align*}
  &J \defeq \{\alpha_1,\cdots,\alpha_{k_1-1},\widehat{{\alpha}_{k_1}},\alpha_{k_1+1},\cdots,\alpha_{k_1+k_2-1},\widehat{{\alpha}_{k_1+k_2}},\alpha_{k_1+k_2+1},\cdots\},\\
  & \tilde{X}\defeq\sum_{\alpha\in J}X_{\alpha}.
  \end{align*}
  Then $\tilde{X}$ has the partition $\mathbf d_2$ and is regular in $\mathfrak l_J$. The Levi subgroup $L_J$ (and its Lie algebra) is of type $\prod_{i=1}^r A_{k_i}$.
  
  \item For $\mathbf{d}=\mathbf{d_2}$, and $\mathbf d_2$ a very even partition. According to the algorithm in the last section, there are only four possible labelings for the last four vertices:
$$
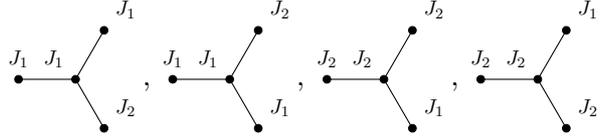
\dynkin[%
edge length=.75cm,
labels*={J_1,J_1,J_1,J_2}]D4, \dynkin[%
edge length=.75cm,
labels*={J_1,J_1,J_2,J_1}]D4, 
\dynkin[%
edge length=.75cm,
labels*={J_2,J_2,J_2,J_1}]D4, 
\dynkin[%
edge length=.75cm,
labels*={J_2,J_2,J_1,J_2}]D4.$$

We only consider the first case, and the other cases follow by applying \cref{flip} and \cref{DiagramAuto}. Write the respective nilpotent element $X=X_{\alpha_{n-2}}+X_{-\alpha_n}+Y$ where $Y = \sum_{\alpha\in J_1\setminus\{\alpha_{n-2}\}}X_{\alpha}+\sum_{\beta\in J_2\setminus\{\alpha_n\}}X_{-\beta}$. Then $Y$ and $X_{-\alpha_n}$ are stabilized by $u_{-\alpha_{n-2}-\alpha_n}(a)$, while $\text{Ad}(u_{-\alpha_{n-2}-\alpha_n}(a))X_{\alpha_{n-2}} = X_{\alpha_{n-2}}+caX_{-\alpha_n}$ for some nonzero constant $c$. Hence, we conjugate $X$ by $u_{-\alpha_{n-2}-\alpha_n}(-1/c)$ gives $\tilde{X}=\sum_{\alpha\in J_1}X_{\alpha}+\sum_{\beta\in J_2\setminus{\alpha_n}}X_{-\beta}$, which lies in the Levi subalgebra $\mathfrak l_J$ with $J=I\setminus\{\alpha_n\}$ and we can use the result in type $A$.
  \end{enumerate}


This completes the proof of \cref{thm:regularity} for type $A$ and $D$.

\subsection{The proof of \cref{theorem:weightthm} for type $A$ and $D$} \label{section:weightproofAD}

\begin{proof}
    We need to show that the $J$ produced by the weight algorithm is a possible output of the \cref{algorithm:partitionA} and \cref{algorithm:partitionD}.
    
    For type $A$, the \cref{algorithm: weight_alg} essentially does not involve Steps 2-5. In this case, the weight of a chunk is exactly the number of vertices in the respective longest path in \cref{algorithm:partitionA}, and the steps of \cref{algorithm: weight_alg} and \cref{algorithm:partitionA} are essentially the same in the view of \ref{constructA}. 
    
    For type $D$, the weight algorithm essentially does not involve Steps 4-6, and it can be checked that the weight so defined will achieve the expected goal.
    
    Indeed, all the algorithms run recursively on subdiagrams, so it suffices to prove for the step when it involves the ``diamond''. We consider the last few vertices in $\Gamma$. Again, by switching $J_1$ and $J_2$ or switching the last two vertices, the proof is similar. We shall call a selection in the weight algorithm \emph{admissible} if the resulting $J$ it will produce is a possible output of \cref{algorithm:partitionD} followed by the construction in \ref{constructD}. In particular, a selection is admissible if it corresponds to a maximal set of disjoint longest paths in the graph $D$ of \cref{algorithm:partitionD}.

{
\centering
\begin{tabular}[H]{| c | c | M{8cm} |}
    \hline
    Case & $\Gamma$ &    \\
    \hline
    1 & $
    \dynkin[%
                    edge length=1cm,
                    labels={,J_1,J_1,J_1,J_1}]D{*.****}
    $
      & 
        This chunk containing the last $m$ vertices has weight $2m-1$. The only critical case is when a type $A$ chunk of rank $2m-2$ is connected to it. In this case, type $A$ chunks appears twice in graph $D$ so type $A$ chunks have the priority. All results in $\mathcal{P}_2$ are admissible.
    \\ \hline

    2 & $
    \dynkin[%
                    edge length=1cm,
                    labels={,J_1,J_1,J_1,J_2}]D{*.****}
    $
      & 

        The graph $D$ can be viewed as two independent copies of type $A$ graphs. All selections in $\mathcal{P}_1$ are admissible.
    \\ \hline

    3 & $
    \dynkin[%
                    edge length=1cm,
                    labels={,J_1,J_1,J_2,J_2}]D{*.****}
    $
      & 

        The only critical case is when the ``diamond'' is involved in tie-breaking in \cref{algorithm:partitionD}, i.e., there is one of the longest paths of length $2$ in the ``diamond'' in the graph $D$. The weights we defined will ensure the selections in $\mathcal{P}_1$ are admissible.
    \\ \hline

    4 & $
    \dynkin[%
                    edge length=1cm,
                    labels={,J_1,J_2,J_1,J_1}]D{*.****}
    $
      & 

        The last two roots are not connected to any dominant chunks. All selections in $\mathcal{P}_1$ are admissible.
    \\ \hline

    5.1 & $
    \dynkin[%
                    edge length=1cm,
                    labels={,J_1,J_1,J_2,J_1,J_2}]D{*.*****}
    $
      & 

        This case requires the rule of the extra root. All selections in $\mathcal{P}_3$ are admissible.
    \\ \hline

    5.2 & $
    \dynkin[%
                    edge length=1cm,
                    labels={,J_2,J_1,J_2,J_1,J_2}]D{*.*****}
    $
      & 

        The $A_2$ chunk $\{\alpha_{n-2},\alpha_n\}$ is not connected to any dominant chunks. All selections in $\mathcal{P}_1$ are admissible.
    \\ \hline

    6 & $
    \dynkin[%
                    edge length=1cm,
                    labels={,J_1,J_2,J_2,J_2}]D{*.****}
    $
      & 

        The chunk of the last three vertices forms a $D_3$ of weight $5$. The only critical case is when a type $A$ chunk of rank $4$ is connected to it. All selections in $\mathcal{P}_2$ are admissible.
    \\ \hline
    
\end{tabular}
}

   We have discussed all possible critical cases when the ``diamond'' part of graph $D$ in \cref{algorithm:partitionD} is involved in a tie-breaking situation, i.e., there are several longest paths connected to each other. Merging with type $A$ parts we see all selections produced by the weight algorithm are admissible. This completes the proof.

\end{proof}

\section{Type $E$} \label{section:E}
In this section, we consider type $E$. Since there is no partition-like classification of nilpotent orbit in exceptional types, we need a different approach. The main idea comes from the treatment for very even partition in type $D$. We provide a way to perform adjoint action to remove one root vector from $X = \sum_{\alpha\in J_1}X_\alpha+\sum_{\beta\in J_2}X_{-\beta}$, according to the local information of a labeled Dynkin diagram. This completes the proof using induction and the result of classical types.

    \subsection{Local Patterns}
    The following lemma becomes handy to cover all cases in type $E$. 
    

  
\begin{lemma} \label{lemma:pattern}
    Let $\mathfrak g$ be a simply-laced reductive Lie algebra. 
\begin{enumerate}
\item Suppose that there exist $\alpha\in J_1$ and $\beta\in J_2$ such that they form $A_2$ as  $$\dynkin[%
parabolic=2,
edge length=.75cm,
labels={\alpha,\beta},
labels*={+,-}]A2$$ and all the roots that are connected to this subdiagram lie in $J_1$. Then $\sum_{\gamma\in J_1}X_\gamma+\sum_{\delta\in J_2}X_{-\delta}$ is conjugate to $\sum_{\gamma\in J_1}X_\gamma+\sum_{\delta\in J_2\setminus\{\beta\}}X_{-\delta}$.

\item Suppose that there exist $\alpha_1,\alpha_2\in J_1$ and $\beta_1,\beta_2\in J_2$ such that they form $A_4$ as  $$\dynkin[%
parabolic=4,
edge length=.75cm,
labels={\alpha_1,\alpha_2,\beta_1,\beta_2},
labels*={+,+,-,-}]A4$$ and all the roots that are connected to this subdiagram lie in $J_1$. Then $\sum_{\gamma\in J_1}X_\gamma+\sum_{\delta\in J_2}X_{-\delta}$ is conjugate to $\sum_{\gamma\in J_1}X_\gamma+\sum_{\delta\in J_2\setminus\{\beta_1\}}X_{-\delta}$.

\item Suppose that there exist $\alpha_1,\alpha_2,\alpha_3\in J_1$ and $\beta_1,\beta_2,\beta_3\in J_2$ such that they form $A_6$ as  $$\dynkin[%
edge length=.75cm,
parabolic=8,
labels={\alpha_1,\alpha_2,\alpha_3,\beta_1,\beta_2,\beta_3},
labels*={+,+,+,-,-,-}]A6$$ and all the roots that are connected to this subdiagram lie in $J_1$. Then $\sum_{\gamma\in J_1}X_\gamma+\sum_{\delta\in J_2}X_{-\delta}$ is conjugate to $\sum_{\gamma\in J_1}X_\gamma+\sum_{\delta\in J_2\setminus\{\beta_1\}}X_{-\delta}$.

\item Suppose that there exist $\alpha_1,\alpha_2,\alpha_3\in J_1$ and $\beta_1,\beta_2,\beta_3\in J_2$ such that they form $E_6$ as  $$\dynkin[%
edge length=.75cm,
parabolic=2,
labels={\beta_1,\beta_2,\beta_3,\alpha_1,\alpha_2,\alpha_3},upside down,
labels*={-,-,-,+,+,+}]E6$$ We assume that no root is connected to $\beta_1$, and all the roots connected this subdiagram lie in $J_1$. Then $\sum_{\gamma\in J_1}X_\gamma+\sum_{\delta\in J_2}X_{-\delta}$ is conjugate to $\sum_{\gamma\in J_1}X_\gamma+\sum_{\delta\in J_2\setminus\{\beta_2\}}X_{-\delta}$.

\item Suppose that $\mathfrak g$ is of type $E_8$, $J_1 = \{\alpha_1,\alpha_2,\alpha_3,\alpha_4\}$, 
 and $J_2 = \{\alpha_5,\alpha_6,\alpha_7,\alpha_8\}$. Then $\sum_{\gamma\in J_1}X_\gamma+\sum_{\delta\in J_2}X_{-\delta}$ is conjugate to $\sum_{\gamma\in J_1}X_\gamma+\sum_{\delta\in J_2\setminus\{\alpha_5\}}X_{-\delta}$.
 $$\dynkin[%
edge length=.75cm,
parabolic=16,
labels={\alpha_1,\alpha_2,\alpha_3,\alpha_4,\alpha_5,\alpha_6,\alpha_7,\alpha_8},upside down,
labels*={+,+,+,+,-,-,-,-}]E8$$
 
\end{enumerate}

Finally, note that if we switch $J_1$ and $J_2$, we obtain the respective results.

\begin{proof}
\begin{enumerate}[noitemsep]

    \item  Consider the adjoint action by a root group element $u_{-\alpha-\beta}(a)$ on each summand. 
         First,
          $$\text{Ad}(u_{-\alpha-\beta}(a))X_\alpha = X_\alpha+caX_{-\beta},$$
          for some nonzero constant $c$. For a positive root $\gamma\neq\alpha,\beta$, $\gamma +k(-\alpha-\beta)$ for $k>0$ is never a root. Hence, $u_{-\alpha-\beta}(a)$ stabilizes all $X_\gamma$. By our assumption, every $\delta\in J_2\setminus \{\beta\}$ is not connected to either $\alpha$ or $\beta$, so $u_{-\alpha-\beta}(a)$ stabilizes all such $X_{-\delta}$. Since we assume that $\mathfrak g$ is simply-laced, $-\alpha-k\beta$ when $k>1$ is not a root, so $u_{-\alpha-\beta}(a)$ stabilizes $X_{-\beta}$. To summarize, 
        $$\text{Ad}(u_{-\alpha-\beta}(a))\left(
        \sum_{\gamma\in J_1}X_\gamma+
        X_{-\beta}+
        \sum_{\delta\in J_2\setminus\{\beta\}}X_{-\delta}
        \right) = 
        \sum_{\gamma\in J_1}X_\gamma+(ca+1)X_{-\beta}+
        \sum_{\delta\in J_2\setminus\{\beta\}}X_{-\delta}.$$
         But then as $c\neq 0$, we can take $a=-1/c$ and this completes the proof.

    \item First apply $\dot s_{\beta_2}$ to $X=\sum_{\gamma\in J_1}X_\gamma+\sum_{\delta\in J_2}X_{-\delta}$ and get $$X_1\defeq c_1X_{\alpha_1}+c_2X_{\alpha_2}+c_3X_{-\beta_1-\beta_2}+c_4X_{\beta_2}+Y,$$ $$Y\defeq \text{Ad}(\dot{s}_{\beta_2})\left(\sum_{\gamma\in J_1\setminus\{\alpha_1,\alpha_2\}}c_\gamma X_\gamma+\sum_{\delta\in J_2\setminus\{\beta_1,\beta_2\}}c_\delta X_{-\delta}\right)$$ where all the $c_i$'s are nonzero.

Apply $u_{-\alpha_2-\beta_1-\beta_2}(a)$ to $X_1$. By our assumption, it stabilizes every root vector that appears except $X_{\alpha_2}$ and $X_{\beta_2}$ as in $(1)$. More explicitly,
$$\text{Ad}(u_{-\alpha_2-\beta_1-\beta_2}(a))X_{\alpha_2} = X_{\alpha_2}+d_1aX_{-\beta_1-\beta_2},$$ and
$$\text{Ad}(u_{-\alpha_2-\beta_1-\beta_2}(a))X_{\beta_2} = X_{\beta_2}+d_2aX_{-\alpha_2-\beta_1},$$ where $d_1$, $d_2$ are nonzero. Hence by setting $a=-c_3/c_2d_1$, we can remove $X_{-\beta_1-\beta_2}$ from the expression and get $$X_2\defeq c_1X_{\alpha_1}+c_2X_{\alpha_2}+c_4X_{\beta_2}-\dfrac{c_3c_4d_2}{c_2d_1}X_{-\alpha_2-\beta_1}+Y.$$ 

Next, we apply $u_{-\alpha_1-\alpha_2-\beta_1}(b)$ to $X_2$. For similar reasons, this element stabilizes every root vector except sending $X_{\alpha_1}$ to $X_{\alpha_1}+d_3yX_{-\alpha_2-\beta_1}$ with $d_3$ nonzero, so we may choose suitable $b$ to remove $X_{-\alpha_2-\beta_1}$ from the expression to get $$X_3\defeq c_1X_{\alpha_1}+c_2X_{\alpha_2}+c_4X_{\beta_2}+Y.$$

Apply $\dot{s}_{\beta_2}$ again. Now $X_3$ is conjugate to $$X_4\defeq \sum_{\gamma\in J_1}c_{\gamma}^{\prime} X_\gamma+\sum_{\delta\in J_2\setminus\{\beta_1\}}c_{\delta}^{\prime}X_{-\delta}$$ where all $c_{\gamma}^{\prime}$ and $c_{\delta}^{\prime}$ are nonzero. 

Finally, by the Lie algebra version of \cref{lem:uniqueness} we can ignore the coefficients, so $X_4$ is conjugate to $\sum_{\gamma\in J_1}X_\gamma+\sum_{\delta\in J_2\setminus\{\beta_1\}}X_{-\delta}$.

\item The conjugation process is similar to (ii), except the Weyl group representative $\dot s_{\beta_2}$ being replaced by $\dot{s}_{\beta_2}\dot{s}_{\beta_3}\dot{s}_{\beta_2}$.

\item  We use the following flow chart to present the conjugation process. The first column records those group elements we use to perform adjoint action. As before we may choose the coefficients so that two root vectors cancel, so we omit the coefficients in the table. The second column records the root vectors that may appear in the expression of the respective nilpotent elements.

\begin{center}
\begin{tabular}{c c}
     & $\alpha_1,\enspace \alpha_2,\enspace\alpha_3,\enspace -\beta_1,\enspace-\beta_2,\enspace-\beta_3$ \\
  $\xrightarrow{\dot{s}_{\beta_1}\dot{s}_{\beta_3}\dot{s}_{\beta_1}}$   & $\alpha_1+\beta_1+\beta_3,\enspace \alpha_2,\enspace \alpha_3,\enspace \beta_1,\enspace -\beta_2,\enspace \beta_3$\\

   $\xrightarrow{u_{-\alpha_1-\beta_1-\beta_2-\beta_3}}$  & $\alpha_1+\beta_1+\beta_3,\enspace \alpha_2,\enspace \alpha_3,\enspace -\alpha_1-\beta_2-\beta_3,\enspace \beta_1,\enspace \beta_3$ \\

   $\xrightarrow{u_{-\alpha_1-\alpha_2-\beta_2-\beta_3}}$  & $\alpha_1+\beta_1+\beta_3,\enspace \alpha_2,\enspace \alpha_3,\enspace -\alpha_1-\alpha_2-\beta_2,\enspace \beta_1,\enspace \beta_3$ \\

   $\xrightarrow{u_{-\alpha_1-\alpha_2-\alpha_3-\beta_2}}$  & $\alpha_1+\beta_1+\beta_3,\enspace \alpha_2,\enspace\alpha_3,\enspace\beta_1,\enspace\beta_3$ \\
   $\xrightarrow{\dot{s}_{\beta_1}\dot{s}_{\beta_3}\dot{s}_{\beta_1}}$  & $\alpha_1,\enspace \alpha_2,\enspace \alpha_3,\enspace -\beta_1,\enspace -\beta_3$ \\

\end{tabular}
\end{center}

Note that there may be other root vectors in the beginning, but by our assumption, every root group element we use here stabilizes them.

\item We will show that $\sum_{\alpha\in J_1}X_\alpha+\sum_{\beta\in J_2}X_{-\beta}$ is conjugate to $\sum_{\alpha\in J}X_{\alpha}$.

We use a format similar to the above, but we simplify the notations further. We denote by $(i_1i_2\cdots i_k)$ the root group element $u_{\alpha_{i_1}+\alpha_{i_2}+\cdots+\alpha_{i_k}}$ in the first column, and for root vectors that may appear in the expression of the respective nilpotent elements in the second column. Also, denote by $-(i_1i_2\cdots i_k)$ correspondingly for the negative root $-\alpha_{i_1}-\alpha_{i_2}-\cdots-\alpha_{i_k}$. Write $\gamma = \alpha_1+\alpha_2+2\sum_{i=3}^7\alpha_i+\alpha_8$, $\delta = \alpha_4+\sum_{i=2}^8\alpha_i$ and $\dot{w}_{678}=\dot{s}_6 \dot{s}_7 \dot{s}_6 \dot{s}_8 \dot{s}_7 \dot{s}_6$.

\begin{center}
\begin{tabular}{c c}
     & $(1), (2), (3), (4), -(5), -(6), -(7), -(8)$ \\
  $\xrightarrow{\dot{s}_{\alpha_6}}$   & $(1), (2), (3), (4), (6), -(8), -(56), -(67)$ \\
   $\xrightarrow{-(456)}$  & $(1), (2), (3), (4), (6),-(8), -(45), -(67)$ \\
   $\xrightarrow{-(245)}$  & $(1), (2), (3), (4), (6), -(8), -(67), -(24567)$ \\

   $\xrightarrow{-(234567)}$  & $(1), (2), (3), (4), (6), -(8), -(67), -(34567), -(2345678)$ \\

   $\xrightarrow{-(134567)}$  & $(1), (2), (3), (4), (6), -(8), -(67), -(1345678), -(2345678), -\gamma$ \\

   $\xrightarrow{-(12345678)}$  & $(1), (2), (3), (4), (6), -(8), -(67), -(2345678), -\gamma$ \\

   $\xrightarrow{-\gamma-\alpha_4}$  & $(1), (2), (3), (4), (6) -(8), -(67), -(2345678)$ \\

    $\xrightarrow{-\delta}$  & $(1), (2), (3), (4), (6), -(8), -(67)$ \\

     $\xrightarrow{\dot{s}_{\alpha_6}}$  & $(1), (2), (3), (4), -(6), -(7), -(8)$ \\

     $\xrightarrow{\dot{w}_{678}}$  & $(1), (2), (3), (4), (6), (7), (8)$ \\
\end{tabular}
\end{center}
\end{enumerate}
\end{proof}
\end{lemma}

\begin{remark}
The same trick in proving this lemma can also be used to calculate the fusion map for types $F_4$ and $G_2$, but we rather use folding.
\end{remark}

\subsection{Detect local patterns}

Now we show that at least one of the local patterns in \cref{lemma:pattern} occurs for a given pair $(J_1,J_2)$ in type $E$.

\begin{notation}
We use $+$ as shorthand for $J_1$ and $-$ for $J_2$ in the labeled Dynkin diagrams.
\end{notation}

\begin{lemma} \label{lemma:detection}
  All the labeled Dynkin diagrams of type $E$ lie in at least one of the situations in \cref{lemma:pattern}
  
  \begin{proof}
      We only prove this Lemma for $E_8$. The proofs for $E_6$ and $E_7$ are similar.
      
      First let us consider the roots $\{\alpha_1,\alpha_3,\alpha_4\}$. There can be $8$ distinct labelings for this $A_3$ subdiagram. By Lemma \ref{flip}, we just need to consider four labelings:
      $$
\dynkin[%
edge length=.75cm,
labels*={-,+,+}]A3,\qquad \dynkin[%
edge length=.75cm,
labels*={+,-,+}]A3,\qquad \dynkin[%
edge length=.75cm,
labels*={+,+,-}]A3,\qquad \dynkin[%
edge length=.75cm,
labels*={+,+,+}]A3. $$

The first two cases are in the situation of \cref{lemma:pattern} (1). Now we start to consider the last two cases, starting by considering the label of $\alpha_2$ and followed by $\alpha_5$, $\alpha_6$, $\alpha_7$, and $\alpha_8$. 

\begin{enumerate}[label=Case \arabic*.]
\item Suppose that $\{\alpha_1,\alpha_3\}\subseteq J_1$ and $\alpha_4\in J_2$.

    \begin{enumerate}[label*=\arabic*.,ref=\theenumi.\arabic*,leftmargin=2em]
        \item Suppose $\alpha_2\in J_2$, and the Dynkin diagram looks like  $
    \dynkin[%
    edge length=.5cm,
    labels={+,-,+,-},upside down]E8.$
    
    No matter what labels $\alpha_5$ has, $\{\alpha_1,\alpha_2,\alpha_3,\alpha_4\}$ forms a subdiagram in the situation of \cref{lemma:pattern}(2).
    
    \item Suppose that $\alpha_2\in J_1$, and the Dynkin diagram looks like $
    \dynkin [%
    edge length=.5cm,
    labels={+,+,+,-},upside down]E8.$
    
    If $\alpha_5\in J_1$, then $\{\alpha_3,\alpha_4\}$ forms a subdiagram in the situation of \cref{lemma:pattern}(1). Now suppose that $\alpha_5\in J_2$. If $\alpha_6\in J_1$, then $\{\alpha_1,\alpha_3,\alpha_4,\alpha_5\}$ form a subdiagram in the situation of \cref{lemma:pattern}(2), so we suppose that $\alpha_6\in J_2$, and list the all the possible labeled Dynkin diagrams as follow. We use $\circ$ instead of $\bullet$ for the vertices in the subdiagrams where the local pattern appears.
    $$
    \dynkin[%
    edge length=.5cm,
    labels={+,+,+,-,-,-,+,+},upside down] E{****oooo},
    \dynkin[%
    edge length=.5cm,
    labels={+,+,+,-,-,-,+,-},upside down] E{******oo},
    \dynkin[%
    edge length=.5cm,
    labels={+,+,+,-,-,-,-,+},upside down] E{******oo},
    \dynkin[%
    edge length=.5cm,
    labels={+,+,+,-,-,-,-,-},upside down] E{oooooo**}$$
    
    \end{enumerate}

\item Suppose that $\{\alpha_1,\alpha_3,\alpha_4\}\subseteq J_1$.

    \begin{enumerate}[label*=\arabic*.,ref=\theenumi.\arabic*,leftmargin=2em]
    \item Suppose $\alpha_2\in J_2$, and the Dynkin diagram looks like $
    \dynkin[%
    edge length=.5cm,
    labels={+,-,+,+},upside down]E8.$
    
    If $\alpha_5\in J_1$, then $\{\alpha_2,\alpha_4\}$ forms a subdiagram in the situation of \cref{lemma:pattern}(1), so we may suppose that $\alpha_5\in J_2$.
    
    \begin{enumerate}[label*=\arabic*.,ref=\theenumi.\arabic*,leftmargin=2em]
        \item Suppose $\alpha_6\in J_1$ and the Dynkin diagram looks like $
    \dynkin[%
    edge length=.5cm,
    labels={+,-,+,+,-,+},upside down]E8.$
    
    If $\alpha_7\in J_1$, then $\{\alpha_5,\alpha_6\}$ forms a subdiagram in the situation of \cref{lemma:pattern}(1), so we may suppose that $\alpha_7\in J_2$. There are two cases. Again, we use $\circ$ instead of $\bullet$ for the vertices in the subdiagrams where the local pattern appears.
    
    $$
    \dynkin[%
    edge length=.5cm,
    labels={+,-,+,+,-,+,-,+},upside down] E{******oo}, \qquad
    \dynkin[%
    edge length=.5cm,
    labels={+,-,+,+,-,+,-,-},upside down] E{*****oo*}$$
    
    \item Suppose $\alpha_6\in J_2$ and the Dynkin diagram looks like $
    \dynkin[%
    edge length=.5cm,
    labels={+,-,+,+,-,-},upside down]E8.$
    \end{enumerate}
    
    If $\alpha_7\in J_1$, then $\{\alpha_1,\alpha_2,\alpha_3,\alpha_4,\alpha_5,\alpha_6\}$ forms a subdiagram in the situation of \cref{lemma:pattern}(4), so we may suppose that $\alpha_7\in J_2$. There are two cases.
    $$
    \dynkin[%
    edge length=.5cm,
    labels={+,-,+,+,-,-,-,+},upside down] E{******oo}, \qquad
    \dynkin[%
    edge length=.5cm,
    labels={+,-,+,+,-,-,-,-},upside down] E{o*ooooo*}$$
    
    \item Suppose $\alpha_2\in J_1$, and the Dynkin diagram looks like $
    \dynkin[%
    edge length=.5cm,
    labels={+,+,+,+},upside down]E8.$
    
    One can similarly easily check that such labeled Dynkin diagrams are in situations of \cref{lemma:pattern}(1)-(3),(5).
    \end{enumerate}
\end{enumerate}
\end{proof}
\end{lemma}

Now we prove \cref{thm:regularity} when $G$ is of type $E$.

\begin{proof}
Combining \cref{lemma:pattern} and \cref{lemma:detection}, we have shown that for a given pair of disjoint subsets $(J_1,J_2)$ of $I$ (in type $E$), and start from $\sum_{\alpha\in J_1}X_\alpha+\sum_{\beta\in J_2}X_{-\beta}$, we can find some suitable group elements and perform the adjoint action to remove one of the root vectors in the summation expression.

After removing one of the root vectors, the nilpotent element lies in a Levi subalgebra which is reductive and of smaller semi-simple rank. Then the proof is completed by induction.
\end{proof}

\subsection{The Proof of \cref{theorem:weightthm} for type $E$} \label{section:weightproofE}
The local patterns iteratively find one of the representatives of the image of a pair $(J_1,J_2)$ under the fusion map, it remains to show this is in the same equivalence class as the representative produced by the weight algorithm. Here we make use of the following proposition which is an easy corollary of \cite[2.12]{LS} and use a computer program \cite{program} to verify this is indeed the case.

\begin{proposition}
If two standard Levi subgroups $L_J$ and $L_{J^\prime}$ are conjugate then there exist sequences of subsets $J_0,J_1,\cdots,J_n$ and $K_0,K_1,\cdots,K_n$ of $I$ with the following properties:
\begin{enumerate}
    \item $J_0 = J$ and $J_n = J^{\prime}$;
    \item $J_{i-1}\cup J_i\subseteq K_i$;
    \item $J_i = w_i\cdot J_{i-1}$, where $w_i$ is the longest element in the parabolic subgroup $W_{K_i}$ of $W$;
    \item $J^{\prime}=w\cdot J$ where $w = \prod_{i=1}^{n}w_i$.
\end{enumerate}
\end{proposition}

\bibliographystyle{unsrt}

\bibliography{Regularity_of_Unipotent_Elements_in_Total_Positivity}

\end{document}